\documentclass{article}

\usepackage{amsmath, bm, dsfont}
\usepackage{amsthm}
\usepackage{amssymb}
\usepackage[capitalise]{cleveref}
\usepackage{enumerate}
\usepackage[colorinlistoftodos]{todonotes}
\usepackage{graphicx}
\usepackage{algorithm}
\usepackage{algpseudocode}
\usepackage{tikz}
\usepackage{pgfplots}
\usetikzlibrary{patterns, shapes.geometric}
\usepackage{subcaption, makecell}
\usepackage{verbatim}
\usepackage{natbib}
\usepackage[margin=1.2in]{geometry}

\newtheorem{theorem}{Theorem}
\newtheorem{lemma}[theorem]{Lemma}
\newtheorem{remark}[theorem]{Remark}
\newtheorem{corollary}[theorem]{Corollary}

\newtheorem{definition}[theorem]{Definition}

\newtheorem{example}[theorem]{Example}

\DeclareMathOperator*{\argmax}{arg\,max}

\DeclareMathOperator*{\sgn}{sgn}

\title{Linear Model Extraction via Factual and Counterfactual Queries}
\author{
Daan Otto \\ \texttt{d.otto@uva.nl} \and
Jannis Kurtz \\ \texttt{j.kurtz@uva.nl} \and
Dick den Hertog \\ \texttt{d.denhertog@uva.nl} \and
Ilker Birbil \\ \texttt{s.i.birbil@uva.nl} \\
\vspace{0.5em}
Amsterdam Business School \\
University of Amsterdam \\
Amsterdam, 1001 NL, The Netherlands
}

\date{}

\begin{document}

\maketitle

\begin{abstract}
In model extraction attacks, the goal is to reveal the parameters of a black-box machine learning model by querying the model for a selected set of data points. Due to an increasing demand for explanations, this may involve counterfactual queries besides the typically considered factual queries. In this work, we consider linear models and three types of queries: factual, counterfactual, and robust counterfactual. First, for an arbitrary set of queries, we derive novel mathematical formulations for the classification regions for which the decision of the unknown model is known, without recovering any of the model parameters. Second, we derive bounds on the number of queries needed to extract the model's parameters for (robust) counterfactual queries under arbitrary norm-based distances. We show that the full model can be recovered using just a single counterfactual query when differentiable distance measures are employed. In contrast, when using polyhedral distances for instance, the number of required queries grows linearly with the dimension of the data space.
For robust counterfactuals, the latter number of queries doubles. Consequently, the applied distance function and robustness of counterfactuals have a significant impact on the model's security. 
\end{abstract}

\section{Introduction}
As machine learning models become progressively more prevalent in research and real-world applications, there is an increasing attention to their impact on privacy, security, and explainability. Recent work outlines attack techniques that threaten model security and privacy, such as model extraction \citep{rigaki2023survey}.
Model extraction aims to reconstruct a target black-box model by querying specific data points and using their outcomes to find the original model's parameters or train a surrogate model that replicates the original model's behavior. Model extraction attacks can jeopardize the model integrity and the intellectual property of the model owner. In combination with model inversion attacks \citep{fredrikson2015model} or attacks that reconstruct training data \citep{boenisch2023curious,ferry2024trained}, this poses privacy risks, which are especially relevant when models are trained on sensitive data, \textit{e.g.}, in medical or financial domains.

Besides, as machine learning models become more complex, their inner workings become less comprehensible to human users. This decrease in explainability leads to the continuously growing field of Explainable Artificial Intelligence (XAI), providing tools for explanations for a large variety of machine learning models, such as counterfactual explanations \citep{wachter2017counterfactual}. A counterfactual for a given factual instance is a (small) perturbation of the instance itself such that the decision of the model flips, answering the question: \textit{``In what situation would the outcome be B instead of A?''}. However, as counterfactual explanations can enhance transparency and trustworthiness, significant security risks are introduced, as the explanations can expose sensitive information about data and the underlying black-box model \citep{shokri2021privacy,nguyen2024survey,milli2019model,Oksuz_2024}. Attackers may exploit counterfactuals to extract the true model parameters  \citep{khouna2025counterfactuals}.

Model extraction methods using factual or counterfactual queries were already studied in several works. \citet{lowd2005adversarial} demonstrated that the parameters of linear classifiers on continuous data can be extracted within an $(1+\epsilon)$ factor using a polynomial number of factual queries. \citet{tramer2016stealing} extend on this and present model extraction attacks for other model classes, including logistic regression, neural networks, and decision trees. The authors consider the machine learning-as-a-service setting, where partial feature vectors can also be queried to obtain confidence values for the model predictions. Besides, \citet{reith2019efficiently} focus on model extraction of Support Vector Regression models using equation-solving attacks. By querying data points, they find a system of equations that find the model parameters. Another approach that has been studied tries to minimize the query cost by framing model extraction as an active learning problem \citep{chandrasekaran2020exploring,pal2020activethief}. 

Other works additionally consider counterfactual queries. In \cite{berning2024trade} and \cite{goethals2023privacy}, the authors describe the privacy issues when the counterfactual mechanism uses an actual existing data point, \textit{e.g.}, another person's data. These works present how a trade-off between privacy and the quality of counterfactual explanations can be made using an anonymization method. A different approach is presented by \citet{aivodji2020model} and consists of creating a dataset of factual instances, then querying a set of counterfactual queries to obtain a balanced dataset to train a surrogate model. In their work, \citet{wang2022dualcf} propose a two-step approach that uses counterfactuals of counterfactuals to obtain the parameters of a linear classifier. Here, they assume the counterfactuals are not lying on the decision boundary. They show the effectiveness of their model experimentally. However, no theoretical bounds on the number of queries needed to extract the original model are presented. \citet{dissanayake2024model} use the $\ell_2$-norm for the \textit{minimal edit} counterfactual to obtain polytope approximations of classifiers whose decision boundary is convex and has a continuous second derivative. The authors also discuss the reconstruction of neural networks with ReLU activation functions using counterfactual queries. 
Lastly, \citet{khouna2025counterfactuals} use counterfactual queries to reconstruct decision trees exactly and provide guarantees on the complexity of the number of queries required for model extraction. However, for many model classes, theoretical bounds on the required number of queries for complete model extraction are still not researched.

In this work, we focus on linear models and consider three types of queries: (i) factuals, (ii) exact counterfactuals, and (iii) exact robust counterfactuals. We study the research questions (i) how much arbitrary queries reveal about the classification regions, and (ii) how many queries are needed to recover the model parameters. By applying techniques from robust optimization, we first derive novel mathematical formulations for the classification regions for the situation where an arbitrary set of query results is given. This extends the current literature to the situation where hand-crafted queries cannot be performed. Second, we show that with a small number of targeted queries the exact model parameters can be recovered. The derived bounds extend the current literature by considering more general setups of counterfactuals, involving arbitrary norm-based distance functions, and robust counterfactuals. Our developed theory shows the effect of the distance function and the robustness of the counterfactuals on the number of required queries for model extraction.   

While the performance of linear models in machine learning is restricted due to their limited complexity, they are widely used because of their interpretability. In highly regulated fields like banking, there are often restrictions on using nonlinear predictors. Regulatory frameworks such as the BCBS \citep{BCBS}, the GDPR’s rules on automated decision-making \citep{GDPR}, and the SR 11-7 Guidance on Model Risk Management \citep{Fed117} emphasize explainability, documentation, and transparency, which favors the adoption of inherently interpretable models like linear regression. Besides, the simple structure of linear models allows us to derive precise mathematical formulations for the subsets of the classification regions and exact bounds on the number of (robust) counterfactual queries needed to extract the model parameters exactly. This analysis provides a foundation that may be extended to more complex classifiers in future work.

Our contributions consist of the following:
\begin{enumerate}
    \item We derive novel and computationally tractable characterizations for the data points for which we can detect the classification without querying the model again, when provided an arbitrary set of (i) factuals, (ii) exact counterfactuals, or (iii) exact robust counterfactuals.
    \item We extend on the current literature by providing upper bounds on the number of (robust) counterfactual queries needed to fully extract the linear classifier's parameters for general norm-based distance functions. Our results show that the choice of the distance measure for (robust) counterfactuals has a significant impact on the number of queries needed for model extraction.
\end{enumerate}

\section{Preliminaries}
In this work, we denote vectors in boldface. Subscripts are used to index vector components, while superscripts are used to index different vectors. We denote the standard basis vectors of $\mathbb{R}^p$ with $\bm{e}^1,\hdots,\bm{e}^p$. We use superscripts in parentheses, $\bm{x}^{(i)}$ to denote data points in a collection $\{\bm{x}^{(i)}\mid i\in I\}$.

Consider a trained classifier $h_{\bm{a},b}: \mathcal X \to \{-1,1\}$ which maps each data point $\bm{x}$ in the data space $\mathcal X\subseteq\mathbb{R}^p$ either to class $-1$ (\textit{`No'}) or $1$ (\textit{`Yes'}). Concretely, we consider linear classifiers given by the hyperplane of the form $\bm{a}^\top \bm{x} -b = 0$ such that
\[
h_{\bm{a},b} (\bm{x}) = \begin{cases}
    1, & \text{ if } \bm{a}^\top \bm{x} - b\ge 0; \\
    -1, & \text{ if } \bm{a}^\top \bm{x} - b< 0,
\end{cases}
\]
where $\bm{a}\in\mathbb{R}^p\setminus\{\bm{0}\}, b\in\mathbb{R}$. We denote $(\bm{a},b)$ as the parameter vector of the linear classifier $h$. Note that classification models involving non-linear transformations of a linear model of the form $f(\bm{a}^\top \bm{x})\ge \alpha$ fall in our framework if $f$ is monotonic and invertible, since in this case we can equivalently reformulate $\bm{a}^\top \bm{x} \ge f^{-1}(\alpha)$ and set $b:=f^{-1}(\alpha)$. Hence, all results in this work can be applied to logistic regression where $f$ is the sigmoid function.

We make the following general assumptions:
\begin{enumerate}
    \item[(A1)] We assume that $\mathcal{X}$ is the full real data space, \textit{i.e.}, $\mathcal{X}=\mathbb{R}^p$.
    \item[(A2)] We assume a non-zero $\bm{a}$, \textit{i.e.}, there exists an $i \in \{ 1,\ldots ,p\}$ with $a_i\neq 0$. Hence, both \textit{`Yes'}  and \textit{`No'} points exist. 
\end{enumerate}

Evidently, two hyperplanes given by $(\bm{a},b)$ and $(\hat{\bm{a}},\hat b)$ are equivalent iff there exists a non-zero scalar $\lambda$ such that $(\bm{a},b) = \lambda (\hat{\bm{a}}, \hat b)$. Hence, hyperplanes are invariant under scaling. 

\begin{definition}
[Equivalent Hyperplane]
Given a hyperplane with parameters $(\bm{a},b)$ such that $\bm{a}^\top \bm{x} -b = 0$, an equivalent hyperplane is given by $(\hat{\bm{a}},\hat b)$ such that $\hat{\bm{a}}^\top \bm{x}-\hat b=0$ if and only if $\bm{a}^\top \bm{x} -b = 0$.
\end{definition}

Next, we define the different query mechanisms we consider in this work, (i) factual, (ii) counterfactual, and (iii) robust counterfactual, which are also depicted in \cref{fig:defs}.

\begin{definition}
[Factual Query]
A factual query $q_F:\mathcal{X}\rightarrow\{0,1\}$ maps a data point $\bm{x}\in X$ to the label of the linear classifier, \textit{i.e.},  $q_F(\bm{x}) := h_{\bm{a},b} (\bm{x})$.
\end{definition}

A mechanism often used in XAI is the counterfactual query. This mechanism outputs a \textit{minimal edit} to a data point to get a desired output from the original model.

\begin{definition}[Counterfactual Query]
Given an arbitrary norm $\| \cdot \|_{N_1}$ (norm-1), a counterfactual (CF) query $q_{CF}:\mathcal{X}\rightarrow\mathcal{X}$ maps a data point $\bm{x}\in \mathcal{X}$ to an optimal solution $\bm{x}_{CF}^*$ of the problem
\begin{equation}\label{eq:CF_problem}
    \begin{aligned}
    \min_{\bm{x}_{CF}} & \ \| \bm{x}_{CF}-\bm{x}\|_{N_1} \\
    s.t. \quad & h_{\bm{a},b}(\bm{x})\neq h_{\bm{a},b} (\bm{x}_{CF}), \\
    & \bm{x}_{CF}\in \mathcal X.
\end{aligned}
\end{equation}
\end{definition}

Note that with slight abuse of notation, we use the minimum operator instead of the infimum operator in the latter problem. Since the region for points which are classified as \textit{`No'} is open, it may happen that the latter optimization problem has no optimal solution. However, in practical settings, a counterfactual is usually calculated over the closure of the region, both for points classified as \textit{`Yes'} or \textit{`No'}. Therefore, when defining the classification regions or calculating (robust) counterfactuals we will consider the closure of the classification regions.

A drawback of counterfactual queries is the lack of robustness; a counterfactual lies on a decision boundary, meaning it is highly sensitive to slight changes in the data point. To combat this problem, we also consider robust counterfactuals which were proposed in the literature, \textit{e.g.}, see \cite{maragno2024finding}.

\begin{definition}[Robust Counterfactual Query]
For a given robustness set $\mathcal S$ and an arbitrary norm  $\| \cdot \|_{N_1}$, a robust counterfactual (RCF) query $q_{RCF}:\mathcal{X}\rightarrow\mathcal{X}$ maps a data point $\bm{x}\in \mathcal{X}$ to an optimal solution $\bm{x}_{RCF}^*$ of the problem
\begin{equation}\label{eq:robust_CF_problem}
\begin{aligned}
    \min_{\bm{x}_{RCF}} & \ \| \bm{x}_{RCF}-\bm{x}\|_{N_1} \\
    s.t. \quad & h_{\bm{a},b}(\bm{x})\neq h_{\bm{a},b} (\bm{x}_{RCF} + \bm{s}) \quad \forall \bm{s}\in\mathcal S, \\
    & \bm{x}_{RCF}\in \mathcal X.
\end{aligned}
\end{equation}
\end{definition}
The definition of a robust counterfactual ensures that for each perturbation of the point by a point in the robustness set $\mathcal S$, the perturbed point remains a counterfactual. Note that a common class of robustness sets is the class of norm-balls of a given radius, where the norm does not have to coincide with the norm used in the objective function of problem \eqref{eq:robust_CF_problem}. To prevent confusion, we use $\|\cdot\|_{N_2}$ (norm-2) to define the robustness set, \textit{i.e.}, $\mathcal S:=\{\bm{s}\mid \|\bm{s}\|_{N_2}\leq \rho\}$ for $\rho>0$. Geometrically, a robust counterfactual is the closest point to the factual instance, such that the whole norm-2-ball around the point lies on the other side of the decision boundary; see Figure \ref{fig:defs}.

In the following, we distinguish between norms that are differentiable at any point $\bm{x}\in\mathcal{X}\setminus\{0\}$ and norms that do not have this property. We will respectively refer to these norms as differentiable norms and non-differentiable norms. Examples of differentiable norms are $\ell_p$-norms with $1<p<\infty$, while $\ell_p$-norms with $p\in \{ 1,\infty\}$ are examples for non-differentiable norms. 

\cref{tab:results} presents a summary of our key results on model extraction using the various query mechanisms introduced above. From this table, we can conclude that using a non-differentiable norm-1 preserves privacy more than a differentiable norm. Moreover, when using robust counterfactuals, we see that more queries are needed to extract the model's parameters than for regular counterfactuals. 

\begin{figure}[H]
   \centering
   \includegraphics[scale=0.4]{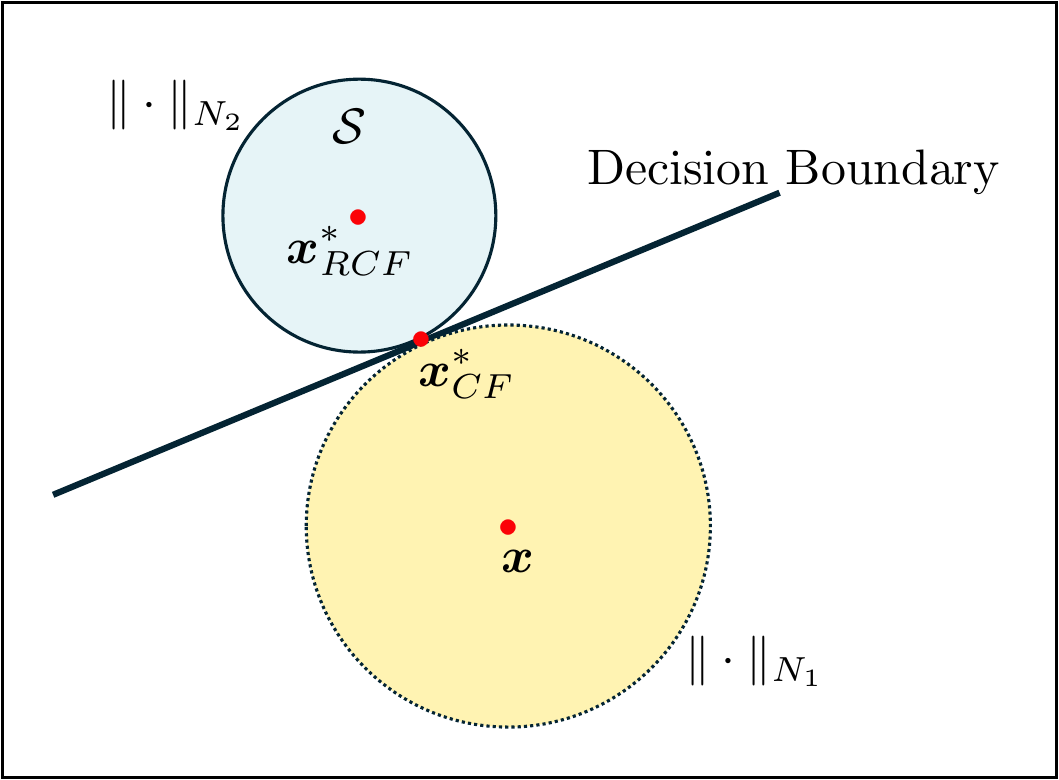}
   \caption{Illustration of the definitions.}
   \label{fig:defs}
\end{figure}

\begin{table}[H]
\begin{center}\resizebox{\textwidth}{!}{%
\begin{tabular}{|c|c||c|c|c|}  \hline
    Query Type & \makecell{Norm-1\\ differentiable} &Reconstruction & $\#$Queries & Result
\\ \hline \hline
    Factual &-& $\epsilon$-approx.& $\mathcal{O}(\log(\epsilon^{-1})+\text{size}(c))$ &\cite{lowd2005adversarial}  \\ \hline
    CF &Yes& Exact& 1 &\cref{thm:CF_theorem_diff}
\\ \hline
CF &No& Exact& $p+1$& \cref{thm:CF_nondiff}
\\ \hline
    RCF  &Yes & Exact& 1 RCF \& 1 Factual & \cref{thm:RCF_diff}
\\ \hline
    RCF  &No  & Exact& $p+1$ RCFs \& $p+1$ Factuals & Corollary \ref{crl:non_diff_RCF_main}\\ \hline
 \end{tabular}%
}
\caption{Information of hyperplane extraction by different query types. \cite{lowd2005adversarial} assume the model parameters have magnitude either $0$ or in $[2^{-c},2^c]$.}
\label{tab:results}
\end{center}
\end{table}

\section{Factual Queries}
In \cite{lowd2005adversarial}, it was shown that by smartly performing up to $\log (\epsilon^{-1})$ factual queries for the classifier $h_{\bm{a},b}$, the hyperplane parameters $\bm{a},b$ can be recovered up to relative accuracy of $\epsilon>0$. However, the latter result relies on the situation that we are able to query a potentially large number of points. On the other hand, if only an arbitrary number of query results is given and we do not have the possibility to perform the required amount of queries, the question remains whether we can already know the query outcome of certain points without querying the model. For example, if we have a set of points given for which we know the model returns the classification $1$, then we know that every point in the convex hull of the given points must be classified as $1$ as well. We next discuss how much information is extracted regarding the \textit{`Yes'} and \textit{`No'} regions by an arbitrary set of factual queried data points and show novel mathematical formulations for these regions which extend beyond the convex hull.

Suppose we have a set of data points $\{\bm{x}^{(i)}, i\in I\}$ where $I$ is the index set. Moreover, assume that the output of the factual query for $\bm{x}^{(i)}$ is \textit{`No'} for $i\in I_{0}$, and is \textit{`Yes'} for $i\in I_1$, where $I_{0} \cup I_1 = I$. We then get the following convex region (polyhedron) as possible values for the unknown parameters $\bm{a},b$:
\begin{equation}\label{eq:Uncertainty_ab_F}
  \mathcal{U}_{\bm{a},b}^F = \{(\bm{a},b) \;\; | \; \;\bm{a}^{\top}\bm{x}^{(i)}-b \leq0  \;  \forall i \in I_{0}, \quad  \bm{a}^{\top}\bm{x}^{(i)} -b\geq 0 \;  \forall i \in I_1 \}.
\end{equation}
We now consider the question: For a given data point $\bar{\bm{x}}\in\mathcal{X}$, can we already know --based on the information given by the factual queries-- whether the classifier yields a \textit{`Yes'} or \textit{`No'}?
Clearly, each point in the convex hull of $\bm{x}^{(i)}, \forall i \in I_{0}$, will yield a \textit{`No'}, and each point in the convex hull of $\bm{x}^{(i)}, \forall i \in I_1$, will yield a \textit{`Yes'}. 

However, we can show that there are many more data points for which we can detect their classification without querying the model again. We refer to all data points for which we can detect the classification will be \textit{`No'} without querying the model again as the \textit{`No'} region ($\mathcal{X}_{\textit{`No'}}$). Similarly, the \textit{`Yes'} region ($\mathcal{X}_{\textit{`Yes'}}$) refers to all data points for which we can detect that the classification will be \textit{`Yes'} without querying the model again. We can test whether we will obtain a \textit{`No'} for $\bar{\bm{x}}$ by solving the following linear optimization problem:
\begin{equation*} \max_{\bm{a},b} \{ \bm{a}^{\top}\bar{\bm{x}} -b\; : \: (\bm{a},b) \in \mathcal{U}_{\bm{a},b}^F  \}.\label{eq:no-rgn}
\end{equation*} 
If the optimal value of this problem is at most $0$, this means that for all possible $(\bm{a},b)\in \mathcal{U}_{\bm{a},b}^F$, the original classifier will output a \textit{`No'} for $\bar{\bm{x}}$. This results in the following convex set for the \textit{`No'} region:
\begin{equation}
  \mathcal{X}_{\textit{`No'}}:=\{ \bm{x}  \;\; | \;\;  \bm{a}^{\top}\bm{x} - b \leq 0 \quad \forall \bm{a},b\in \mathcal{U}_{\bm{a},b}^F \}=\{ \bm{x}  \;\; | \;\;  \max_{\bm{a},b\in \mathcal{U}_{\bm{a},b}^F}\bm{a}^{\top}\bm{x} - b \leq 0  \}.\label{eq:no-rgn-opt}
\end{equation}
On the other hand, if there exists an $(\bm{a},b) \in \mathcal{U}_{\bm{a},b}^F$ such that $\bm{a}^{\top}\bar{\bm{x}} -b> 0$ we cannot know whether $\bar{\bm{x}}$ will be classified as \textit{`No'} without factual querying $\bar{\bm{x}}$.
Similarly, if the optimal value of the following linear optimization problem
\begin{equation*}
    \min_{\bm{a},b} \{ \bm{a}^{\top}\bar{\bm x} -b\; : \: (\bm{a},b) \in \mathcal{U}_{\bm{a},b}^F  \}\label{eq:yes-rgn}
\end{equation*} 
is larger than $0$, the original classifier will output \textit{`Yes'} for $\bar{\bm x}$. Moreover, note that the set of points, for which we know for sure that we will obtain a \textit{`Yes'} from the original classifier, is the following convex region: 
\begin{equation}
  \mathcal{X}_{\textit{`Yes'}}:=\{ \bm{x}  \;\; | \;\;  \bm{a}^{\top}\bm{x} - b \geq 0 \quad \forall \bm{a},b\in \mathcal{U}_{\bm{a},b}^F \}=\{ \bm{x}  \;\; | \;\;  \min_{\bm{a},b\in \mathcal{U}_{\bm{a},b}^F}\bm{a}^{\top}\bm{x} - b \geq 0  \}.\label{eq:yes-rgn-opt}
\end{equation}

By dualizing the optimization problems described in \eqref{eq:no-rgn-opt} and \eqref{eq:yes-rgn-opt}, we get different equivalent formulations of the \textit{`Yes'} and \textit{`No'} regions. In the following theorem, we derive such a mathematical formulation for the regions for which we can detect the classification without additional queries.
\begin{theorem}\label{thm:rgn_factuals}
    Given data points $\bm{x}^{(i)}$, $i\in I$ such that each $\bm{x}^{(i)}$ is classified as \textit{`No'} for $i\in I_{0}$ and as \textit{`Yes'} for $i\in I_1$, where $I_{0} \cup I_1 = I$. Then, the \textit{`No'} and \textit{`Yes'} regions are given by
    \begin{align*}
    \mathcal{X}_{\textit{`No'}} &= \left\{ \bm{x} \; | \;\; \exists \bm{u} \; : \;
\sum_{i\in I_{0}} u_i- \sum_{i\in I_1} u_i = 1, \; 
\sum_{i\in I_{0}} \bm{x}^{(i)}u_i- \sum_{i\in I_1} \bm{x}^{(i)}u_i = \bm{x}, \; 
 \bm{u}\geq \bm{0}\right\},
     \end{align*} 
and
\begin{align*}
    \mathcal{X}_{\textit{`Yes'}}&=\left\{ \bm{x} \; | \;\; \exists \bm{u} \; : \;
\sum_{i\in I_1} u_i- \sum_{i\in I_{0}} u_i = 1, \; 
\sum_{i\in I_1} \bm{x}^{(i)}u_i- \sum_{i\in I_{0}} \bm{x}^{(i)}u_i = \bm{x}, \; 
 \bm{u}\geq \bm{0}\right\},
    \end{align*} 
respectively.
\end{theorem}

The results of \cref{thm:rgn_factuals} show that finding out whether a data point $\bar{\bm{x}}$ is in $\mathcal{X}_{\textit{`No'}}$ or $\mathcal{X}_{\textit{`Yes'}}$ is computationally tractable since this can be done by optimizing a trivial objective function over the feasible set described in \cref{thm:rgn_factuals} resulting in a linear optimization problem which can be solved by state-of-the-art optimization solvers efficiently. Moreover, we notice that if we set $u_i=0$, $\forall i\in I_1$, then the \textit{`No'} region boils down to the convex hull of the points $\bm{x}^{(i)}$, $i\in I_{0}$. An example of \textit{`Yes'} and \textit{`No'} regions is depicted in \cref{fig:ex_rgns_F}.
\begin{figure}[h]
    \centering
    \includegraphics[width=0.6\linewidth]{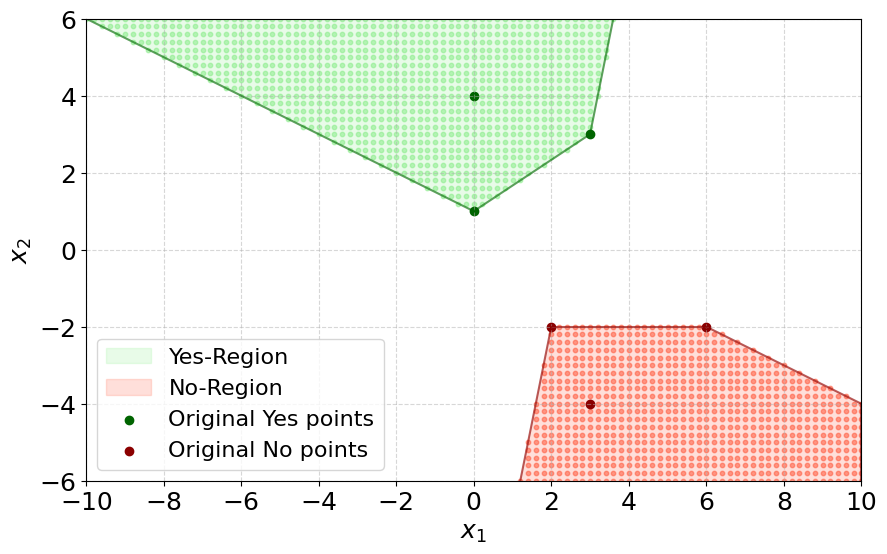}
    \caption{Example of \textit{`Yes'} and \textit{`No'} regions for a set of given factuals.}
        \label{fig:ex_rgns_F}
\end{figure}


\section{Counterfactual Queries}\label{sec:CF}
In this section, we first discuss how much information is extracted about the \textit{‘Yes’} and \textit{‘No’} regions by a set of factual and counterfactual data points. Afterwards, we examine how many counterfactual queries are needed to retrieve the original hyperplane exactly.

\subsection{Classification Regions}\label{sec:regions_CF}
Suppose that, besides factual queries, we also get (\textit{minimal edit}) counterfactuals for some points of our dataset. Concretely, for points $\bm{x}^{(j)}$ from an index subset $j\in J_{0}\cup J_1=J\subseteq I$ we get a counterfactual $\bm{x}_{CF}^{(j)}:=q_{CF}(\bm{x}^{(j)})$. Suppose the \textit{minimal edit} is with respect to a certain norm $\| \cdot\|_{N_1}$, and let $\rho_j = \| \bm{x}^{(j)} - \bm{x}_{CF}^{(j)}\|_{N_1}$. Then, we know that each point $\bm{z}$ in the ball with center $\bm{x}^{(j)}$ and radius $\rho_j$ will be classified similarly to $\bm{x}^{(j)}$, \textit{i.e.}, this ball lies on the same side of the hyperplane as $\bm{x}^{(j)}$. Hence, we know 
\begin{align}
    \bm{a}^{\top}\bm{z} -b&\leq 0  \qquad \forall \bm{z} : \| \bm{z} - \bm{x}^{(j)}\|_{N_1} \leq \rho_j \; \; &\forall j \in J_{0},\label{eq:CF_J0_0}\\
    \bm{a}^{\top}\bm{z} -b &\geq 0   \qquad 
  \forall \bm{z} : \| \bm{z} - \bm{x}^{(j)}\|_{N_1} \leq \rho_j&\; \; \forall j \in J_1.\label{eq:CF_J1_0}
\end{align}
Moreover, we know that a counterfactual has a classification that is opposite to its corresponding factual. Concretely, this means
\begin{align}
  \bm{a}^{\top}\bm{x}_{CF}^{(j)} -b &\geq 0  &\quad \forall j \in J_{0},  \label{eq:CF_J0_1} \\
  \bm{a}^{\top}\bm{x}_{CF}^{(j)} -b&\leq 0 \;  &\quad \forall j \in J_1.
  \label{eq:CF_J1_1}
\end{align}
Since constraints (\ref{eq:CF_J0_0}) and (\ref{eq:CF_J1_0}) also hold for $\bm{z}=\bm{x}_{CF}^{(j)}$, the inequalities in (\ref{eq:CF_J0_1}) and (\ref{eq:CF_J1_1}) can be replaced by equalities. With this extra information, we can characterize the uncertainty set of $\bm{a},b$ as $\mathcal{U}_{\bm{a},b}^{CF}$ given by the following set of constraints:
\begin{align*}
  \bm{a}^{\top}\bm{x}^{(i)} -b&\leq 0  &\quad \forall i \in I_{0},  \notag\\
  \bm{a}^{\top}\bm{x}^{(i)} -b&\geq 0  &\quad \forall i \in I_1, \notag\\
  \bm{a}^{\top}\bm{z} -b&\leq 0  \;\quad \forall \bm{z} : \| \bm{z} - \bm{x}^{(j)}\|_{N_1} \leq \rho_j  &\quad\forall j \in J_{0}, \\ 
  \bm{a}^{\top}\bm{z} -b &\geq 0 \;  \quad 
  \forall \bm{z} : \| \bm{z} - \bm{x}^{(j)}\|_{N_1} \leq \rho_j&\quad \forall j \in J_1, \\
  \bm{a}^{\top}\bm{x}_{CF}^{(j)} -b&= 0   &\quad \forall j \in J.
\end{align*}
Similar to the case with the factuals, we can write the \textit{`Yes'} and \textit{`No'} region as
\begin{align*}
    \mathcal{X}_{\textit{`Yes'}}&=\{ \bm{x}  \;\; | \;\;  \min_{\bm{a},b\in \mathcal{U}_{\bm{a},b}^{CF}}\bm{a}^{\top}\bm{x} - b \geq 0  \} &\text{and}&&
  \mathcal{X}_{\textit{`No'}}&=\{ \bm{x}  \;\; | \;\;  \max_{\bm{a},b\in \mathcal{U}_{\bm{a},b}^{CF}}\bm{a}^{\top}\bm{x} - b \leq 0  \},
\end{align*}
respectively.

Using this new formulation of the uncertainty set $\mathcal{U}_{\bm{a},b}^{CF}$, we need perspective functions to dualize the inner optimization problem to obtain a dual characterization for the \textit{`Yes'}  and \textit{`No'} regions. 

\begin{theorem}\label{thm:region-CF}
Consider a dataset of points $\bm{x}^{(i)}$ for $i\in I$ such that $q_F(\bm{x}^{(i)})=-1$ for all $i\in I_{0}$ and $q_F(\bm{x}^{(i)})=1$ for all $i\in I_1=I\setminus I_{0}$. Moreover, consider the points $\bm{x}_{CF}^{(j)}$ for $j\in J\subseteq I$ such that $\bm{x}_{CF}^{(j)}=q_{CF}(\bm{x}^{(j)})$. Let $J_{0}\subseteq I_{0}, J_1\subseteq I_1,$ and $J=J_{0}\cup J_1$. Then, the \textit{`No'}  and \textit{`Yes'} regions are characterized by the conic quadratic sets
\[ 
\mathcal{X}_{\textit{`No'}} =
\left\{ \begin{array}{lll}
\bm{x}  \mid \exists \bm{t},\bm{u},\bm{v},\bm{y}:& 
\displaystyle-\sum_{i\in I_{0}}t_{i} + \displaystyle\sum_{i\in I_1}t_{i}
-\displaystyle\sum_{j\in J}y_{b}
-\displaystyle\sum_{j\in J}v_{j} &= -1,\\  
&\displaystyle\sum_{i\in I_{0}}t_{i}\bm{x}^{(i)} - \displaystyle\sum_{i\in I_1}t_{i}\bm{x}^{(i)}+\displaystyle\sum_{j\in J}\bm{y}_a^{(j)}+\displaystyle\sum_{j\in J}v_{j}\bm{x}_{CF}^{(j)} &= \bm{x},\\
&u_j(\| \bm{y}_a^{(j)}/u_j - \bm{x}^{(j)}\|_{N_1}  - \rho_j)\leq  0&\forall j\in J_{0}, \\
 &u_j (y_b^{(j)}/u_j + 1)=0&\forall j\in J_{0},\\
 &u_j(\| \bm{y}_a^{(j)}/u_j + \bm{x}^{(j)}\|_{N_1} - \rho_j)\leq  0&\forall j\in J_1,\\
 &u_j (y_b^{(j)}/u_j - 1)=0&\forall j\in J_1,\\
 &\bm{y}^{(j)}=(\bm{y}_a^{(j)},y_b^{(j)})\in\mathbb{R}^{p+1}&\forall j\in J\\
 & \bm{t}\in\mathbb{R}^{|I|}_{\geq 0},\;\bm{u}\in\mathbb{R}^{|J|}_{\geq 0},\; \bm{v}\in\mathbb{R}^{|J|}.
 \end{array}\right\}
\]
and 
\[ \mathcal{X}_{\textit{`Yes'}} =
\left\{ \begin{array}{lll}
\bm{x}  \mid \exists \bm{t},\bm{u},\bm{v},\bm{y}:& 
\displaystyle-\sum_{i\in I_{0}}t_{i} + \displaystyle\sum_{i\in I_1}t_{i}
-\displaystyle\sum_{j\in J}y_{b}
-\displaystyle\sum_{j\in J}v_{j} &= 1,\\  
&\displaystyle\sum_{i\in I_{0}}t_{i}\bm{x}^{(i)} - \displaystyle\sum_{i\in I_1}t_{i}\bm{x}^{(i)}+\displaystyle\sum_{j\in J}\bm{y}_a^{(j)}+\displaystyle\sum_{j\in J}v_{j}\bm{x}_{CF}^{(j)} &= -\bm{x},\\
&u_j(\| \bm{y}_a^{(j)}/u_j - \bm{x}^{(j)}\|_{N_1}  - \rho_j)\leq  0&\forall j\in J_{0}, \\
 &u_j (y_b^{(j)}/u_j + 1)=0&\forall j\in J_{0},\\
 &u_j(\| \bm{y}_a^{(j)}/u_j + \bm{x}^{(j)}\|_{N_1} - \rho_j)\leq  0&\forall j\in J_1,\\
 &u_j (y_b^{(j)}/u_j - 1)=0&\forall j\in J_1,\\
 &\bm{y}^{(j)}=(\bm{y}_a^{(j)},y_b^{(j)})\in\mathbb{R}^{p+1}&\forall j\in J\\
 & \bm{t}\in\mathbb{R}^{|I|}_{\geq 0},\;\bm{u}\in\mathbb{R}^{|J|}_{\geq 0},\; \bm{v}\in\mathbb{R}^{|J|}.
 \end{array}\right\},
\]
respectively.
\end{theorem}
\cref{thm:region-CF} shows that finding out whether a data point $\bar{\bm{x}}$ is in $\mathcal{X}_{\textit{`No'}}$ or $\mathcal{X}_{\textit{`Yes'}}$ is computationally tractable, because of its conic quadratic formulation. To find out whether a data point $\bar{\bm{x}}$ is in $\mathcal{X}_{\textit{`No'}}$ or $\mathcal{X}_{\textit{`Yes'}}$, we can optimize a trivial objective function over the set of constraints as described in \cref{thm:region-CF}. The resulting conic quadratic optimization problem can be efficiently solved using state-of-the-art solvers. In Figure \ref{fig:ex_rgns_CF}, we show a two-dimensional visualization of the classification regions when there is one data point and a corresponding counterfactual, \textit{i.e.}, $|I|=|J|=1$. We consider three cases, when the counterfactual mechanism uses the (i) $\ell_1$-norm, (ii) $\ell_2$-norm, and (iii) $\ell_{\infty}$-norm. We see that for the non-differentiable norms ($\ell_1,\ell_{\infty}$) there are still areas for which we cannot conclude the classification. For the differentiable $\ell_2$-norm, however, it seems only one factual and corresponding counterfactual extracts the whole model. In the next section, we verify this hypothesis.

\begin{figure}[h]
     \centering
     \begin{subfigure}[b]{0.32\textwidth}
         \centering
         \includegraphics[width=\textwidth]{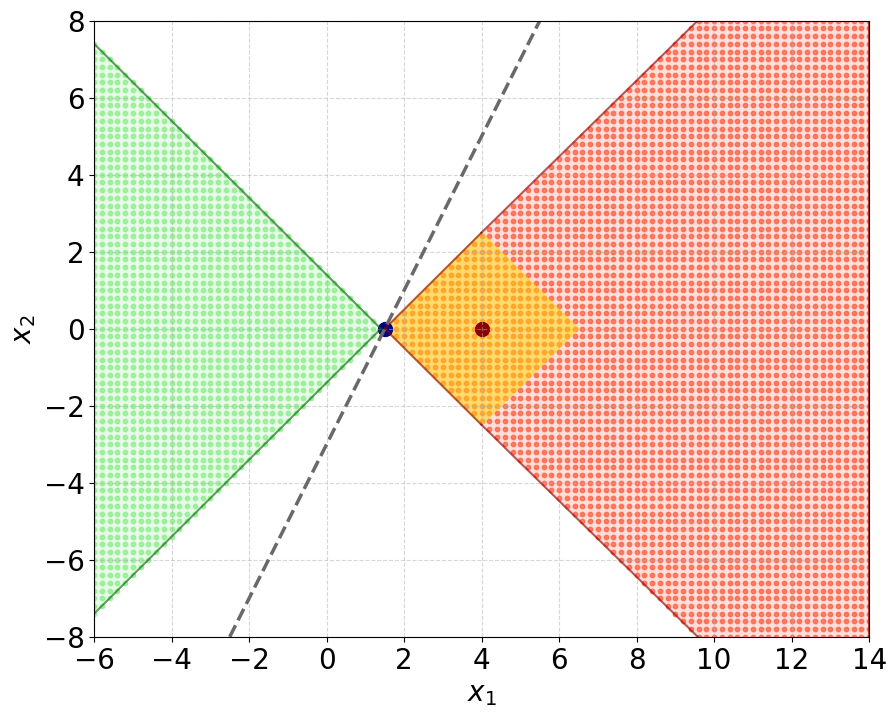}
         \caption{$N_1=\ell_1$.}
         \label{fig:ex_rgns_CF_L1}
     \end{subfigure}
     \hfill
     \begin{subfigure}[b]{0.32\textwidth}
         \centering
         \includegraphics[width=\textwidth]{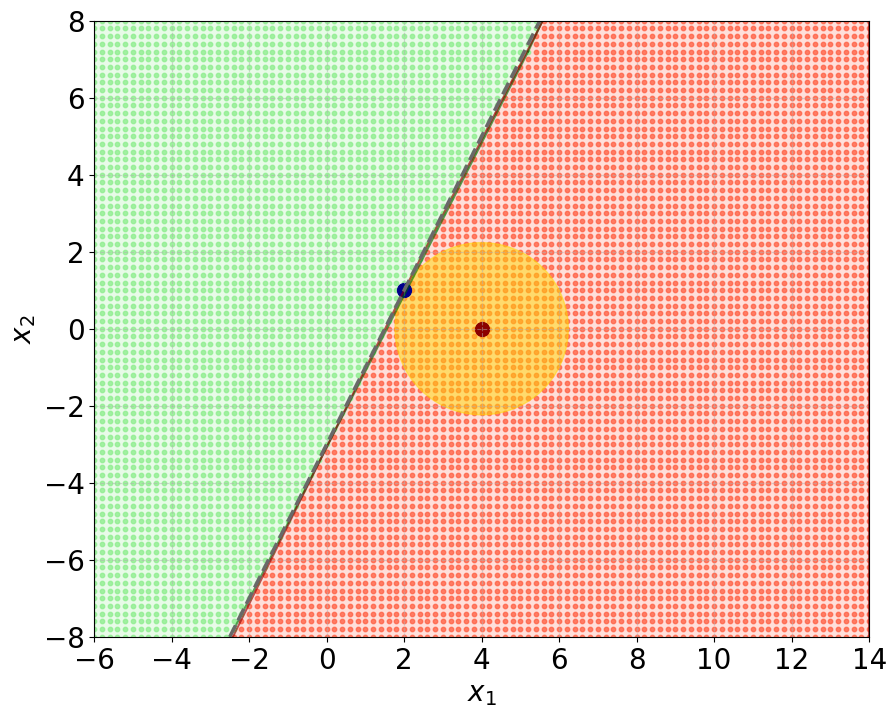}
         \caption{$N_1=\ell_2$.}
         \label{fig:ex_rgns_CF_L2}
     \end{subfigure}
     \hfill
     \begin{subfigure}[b]{0.32\textwidth}
         \centering
         \includegraphics[width=\textwidth]{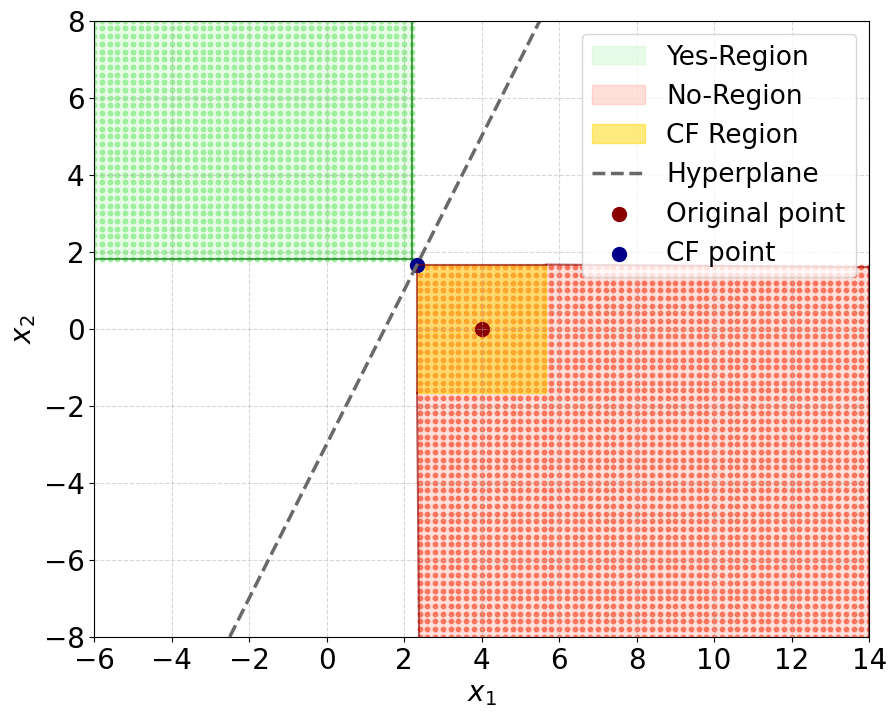}
         \caption{$N_1=\ell_{\infty}$.}
         \label{fig:ex_rgns_CF_Linf}
     \end{subfigure}
        \caption{Example of the classification regions given one data point classified as \textit{`No'} and corresponding counterfactual for different choices of norm-1.}
        \label{fig:ex_rgns_CF}
\end{figure}

\subsection{Extracting the Hyperplane}
We will show in this section that the number of counterfactual queries needed to recover the linear classifier $h_{\bm{a},b}$ depends on the norm $\| \cdot \|_{N_1}$ used in problem \eqref{eq:CF_problem}. We first use classical optimality conditions to derive general conditions on the linear hyperplane parameters $\bm{a}$ and $b$.

In the following we denote by $\partial f(\bm{x}_0)$ the subdifferential set of a function $f: \mathbb{R}^p \to \mathbb{R}$ at point $\bm{x}_0$, \textit{i.e.},
\[
\partial f(\bm{x}_0):= \{ \bm{v}\in\mathbb{R}^p: f(\bm{x}_0) - f(\bm{x}) \ge \bm{v}^\top (\bm{x}_0-\bm{x}) \ \forall \ \bm{x}\in \mathcal X\}.
\]
\begin{lemma}\label{lem:optimality_conditions_CF}
Let $\bm{x}_F\in \mathcal X$ be an arbitrary factual instance and $q_{CF}(\bm{x}_F) = \bm{x}_{CF}^*$ with $\bm{x}_{CF}^*\neq \bm{x}_F$ be its corresponding optimal counterfactual under an arbitrary norm $f(\bm{x}) = \| \bm{x} \|_{N_1}$. Then, there exists a scalar $\lambda^*\in\mathbb R\setminus \{ 0\}$ such that
\begin{align*}
    &\bm{a}^\top \bm{x}_{CF}^* = b, \\
    &\lambda^* \bm{a} \in \partial  f(\bm{x}_{CF}^* - \bm{x}_F).
\end{align*}
\end{lemma}

The latter lemma indicates that if $f$ is differentiable, then $\partial f$ is a singleton, which provides us the direction of $\bm{a}$, while for non-differentiable norms $\partial f$ may be an infinite set, concealing the true direction of $\bm{a}$. We analyze both cases in the following two subsections.

\subsubsection{Differentiable Norms}
For differentiable norms, the subdifferential set in Lemma \ref{lem:optimality_conditions_CF} contains only the gradient of $f$. Hence, in this case $\bm{a}$ can directly be extracted, which shows that we only need one counterfactual query to recover the hyperplane of the classifier.

\begin{theorem}\label{thm:CF_theorem_diff}
Let $\bm{x}_F\in \mathcal X$ be an arbitrary factual instance and $q_{CF}(\bm{x}_F) = \bm{x}_{CF}^*$ with $\bm{x}_{CF}^*\neq \bm{x}_F$ be its corresponding optimal counterfactual under an arbitrary differentiable norm $f(\bm{x}) = \| \bm{x} \|_{N_1}$.  This one counterfactual query is enough to extract the original classifier's parameters since for $\hat{\bm{a}} = \nabla f(\bm{x}_{CF}^*-\bm{x}_F)$ and $\hat b = \hat{\bm{a}}^\top \bm{x}_{CF}^*$ it holds that the hyperplane given by $\hat{\bm{a}},\hat{b}$ is equivalent to the original hyperplane with parameters $\bm{a},b$.
\end{theorem}
Note that \cref{thm:CF_theorem_diff} finds a hyperplane equivalent to the one used by the original classifier, however to obtain an equivalent classifier, we still need to know which side of the hyperplane is classified as $1$ and which one as $-1$. This can be done by one factual query of a point lying outside of the hyperplane. Note that the assumption $\bm{x}_F\neq \bm{x}_{CF}^*$ in \cref{thm:CF_theorem_diff} is not too restrictive since it is only violated if $\bm{x}_F$ lies on the hyperplane. However, for a randomly drawn point $\bm{x}_F\in \mathcal X$, this happens with probability zero. In case we have a point $\bm{x}_F$ which lies on the hyperplane, we may query $p$ linearly independent points around $\bm{x}_F$, \textit{e.g.}, the points $\bm{x}_F + \bm{e}^i$ for $i=1,\ldots ,p$. Then at least one of these points does not lie on the hyperplane, and we can apply Theorem \ref{thm:CF_theorem_diff} to recover the hyperplane. 

\subsubsection{Non-differentiable norms}
Unfortunately, Lemma \ref{lem:optimality_conditions_CF}  implies that, in general, one counterfactual query does not extract the original hyperplane if the norm $\|\cdot\|_{N_1}$ is non-differentiable. This is because the subdifferential for non-differentiable norms is not necessarily a singleton but an infinite set. Figures \ref{fig:ex_rgns_CF_L1} and \ref{fig:ex_rgns_CF_Linf} show, in the case of non-differentiable norms, that the uncertainty in $\bm{a},b$ leads to regions for which the classification is not known.
In this case, the subdifferential set of the norm-function $f(\bm{x})=\|\bm{x}\|_{N_1}$ in Lemma \ref{lem:optimality_conditions_CF} can be reformulated as
\begin{equation*}\label{eq:reformulation_subdifferentiable_norm}
    \partial  f(\bm{x}) = \{ \bm{v}\in \mathbb R^p: \bm{v}^\top \bm{x} = \|\bm{x}\|_{N_1} , \|\bm{v}\|_{N_1}^* \le 1\},
\end{equation*}
where $\| \cdot \|_{N_1}^*$ is the dual norm; see \cite{rockafellar1997convex}. This leads to the following result. 
\begin{corollary}
\label{cor:optimality_conditions_CF_non_differentiable_norm}
Let $\bm{x}_F\in \mathcal X$ be an arbitrary factual instance and $q_{CF}(\bm{x}_F) = \bm{x}_{CF}^*$ under an arbitrary norm $ \| \cdot \|_{N_1}$.  Then, $\bm{x}_{CF}^*$ is an optimal counterfactual if and only if there exists a $\lambda\in\mathbb R\setminus \{ 0\}$ such that
\begin{align*}
    &\bm{a}^\top \bm{x}_{CF}^* = b, \\
    &  \lambda \bm{a}^\top (\bm{x}_{CF}^* - \bm{x}_F) = \|\bm{x}_{CF}^* - \bm{x}_F\|_{N_1}, \\
    & \|\lambda \bm{a} \|_{N_1}^* \le 1 .
\end{align*}    
\end{corollary}
 To extract the original hyperplane, we need additional counterfactual queries. We note that a set of $p$ linearly independent points $\{\bm{x}^{(1)},\hdots, \bm{x}^{(p)}\}$ on the hyperplane, \textit{i.e.},
\begin{align}\label{eq:syst_eq_CF}
\bm{a}^\top \bm{x}^{(i)} - b & = 0 \quad \forall i=1,\ldots, p
\end{align}
is enough to extract $\bm{a},b$. Thus, by finding $p$ linearly independent counterfactual points, we can extract a hyperplane revealing the original classifier. 

\begin{remark}
In the case that one of our counterfactual queries returns $\bm{0}$ as a counterfactual for a certain point, we know that $b=0$ since $\bm{0}$ lies on the hyperplane. Therefore, the solution space of $\bm{a},b$ boils down to the solution space of $\bm{a}$, which is one-dimensional when we have $p-1$ linearly independent points $\{\bm{x}^{(1)},\ldots, \bm{x}^{(p-1)}\}$. The system of equations described in \eqref{eq:syst_eq_CF} then simplifies to
\begin{align*}
\bm{a}^\top \bm{x}^{(i)} & = 0 \quad \forall i=1,\ldots, p-1.
\end{align*}
Since we only need the direction of $\bm{a}$ (any scaling leads to the same hyperplane), the latter equation system allows us to recover the classifier. Thus, if $\bm{0}$ is a counterfactual, then by finding $p-1$ linearly independent counterfactual vectors, the optimality conditions in Corollary \ref{cor:optimality_conditions_CF_non_differentiable_norm} still ensure that we retrieve a hyperplane equivalent to the original classifier.
\end{remark}

To obtain a set of $p$ linearly independent points on the hyperplane, we strategically choose the factual instances to query. The following two lemmas are helping us with that.

\begin{lemma}\label{lem:CF_direction}
Let $\bm{x}_F\in \mathcal X$ be an arbitrary factual instance that does not lie on the hyperplane and $q_{CF}(\bm{x}_F) = \bm{x}_{CF}^*$ under an arbitrary norm $ \| \cdot \|_{N_1}$. Then,
\begin{equation}\label{eq:CF_direction}
        \bm{x}_{CF}^*=\bm{x}_F+d_{\bm{x}_F}\bm{v},
    \end{equation}
with $d_{\bm{x}_F}=\frac{b-\bm{a}^{\top}\bm{x}_F}{\|\bm{a}\|_{N_1}^*}$ is an optimal counterfactual if and only if $\|\bm{v}\|_{N_1}\leq 1$ and $\bm{a}^{\top}\bm{v}=\|\bm{a}\|_{N_1}^*$.
\end{lemma}

The lemma demonstrates the existence of a direction $\bm{v}$ such that, for any factual point, there is a corresponding counterfactual point obtained by perturbing the factual point in the direction of $\bm{v}$. Especially, the direction is independent of the factual instance. From the definition of the dual norm it follows that
\begin{equation}\label{eq:max_dir}
    \|\bm{a}\|_{N_1}^* = \max_{\| \bm{v}\|_{N_1} \le 1} \bm{a}^\top \bm{v},
\end{equation}
and hence, we know that the direction $\bm{v}$ must be a maximizer of the latter problem. 

If multiple optimal counterfactuals exist, one counterfactual query might perturb the original point in direction $\bm{v}^{(1)}:=\bm{x}^{(1)*}_{CF}-\bm{x}_F^{(1)}$, while another counterfactual query might perturb the original point in direction $\bm{v}^{(2)}=\bm{x}^{(2)*}_{CF}-\bm{x}_F^{(2)}$ with $\bm{v}^{(1)} \neq \bm{v}^{(2)}$. A consequence of Lemma \ref{lem:CF_direction} is that if such multiple optimal counterfactuals exist, we can always retrieve a counterfactual in the direction of $\bm{v}^{(1)}$. Concretely if for two factual points $\bm{x}_F^{(1)}, \bm{x}_F^{(2)} \in \mathcal X$ we have
\[
\bm{x}_{CF}^{(1)*}
  = \bm{x}_F^{(1)} +  \bm{v}^{(1)}
  \quad\text{and}\quad
\bm{x}_{CF}^{(2)*}
  = \bm{x}_F^{(2)} +  \bm{v}^{(2)},
\]
where $\bm{v}^{(1)} \neq \bm{v}^{(2)}$ and $\bm{v}^{(1)}, \bm{v}^{(2)} \neq \bm{0}$. Then
\[
\bar{\bm{x}}_{CF}^{(1)*}
  = \bm{x}_F^{(1)} + \frac{\|\bm{v}^{(1)}\|_{N_1}}{\|\bm{v}^{(2)}\|_{N_1}} \bm{v}^{(2)}
  \quad\text{and}\quad
\bar{\bm{x}}_{CF}^{(2)*}
  = \bm{x}_F^{(2)} +  \frac{\|\bm{v}^{(2)}\|_{N_1}}{\|\bm{v}^{(1)}\|_{N_1}} \bm{v}^{(1)}
\]
are also optimal counterfactuals. This follows from Lemma \ref{lem:CF_direction} because if for one factual point there exists a counterfactual in the direction of $\bm{v}^{(1)}$, then for any point there exists a counterfactual in the direction of $\bm{v}^{(1)}$. If for another factual point, there exists a counterfactual in a different direction, we obtain the distance between the counterfactual and the factual. To obtain a counterfactual in the direction of $\bm{v}^{(1)}$, we only have to rescale the perturbation to match this distance between the factual and the counterfactual.

\begin{example}[$\ell_1$ and $\ell_{\infty}$ norm]\label{example:l1-linf}
For both the $\ell_1$ and $\ell_{\infty}$ norms, a maximizer of \eqref{eq:max_dir} over the corresponding norm ball is attained at a vertex of the feasible region. Hence, for every factual instance $\bm{x}_F\in\mathcal{X}$ there exists an optimal counterfactual of the form $\bm{x}_{CF}^* = \bm{x}_F + d_{\bm{x}_F}\bm{v}$,
where $\bm{v}$ is a vertex of the norm ball.

For $N_1=\ell_1$, this means that there exists an optimal counterfactual that modifies only one coordinate $j_{0}$ of $\bm{x}_F$. This becomes clear when choosing $\bm{v} = \sgn(a_{j_{0}})\,\bm{e}^{j_{0}}$, where $j_{0} \in \argmax_{j=1,\ldots,p} |a_j|$. Then, $\|\bm{v}\|_1 = 1$ holds and $\bm{a}^\top\bm{v} = |a_{j_{0}}|= \|\bm{a}\|_{\infty} = \|\bm{a}\|_{1}^*$. However, the counterfactual query does not need to return a vertex solution when multiple optimal counterfactuals exist, \textit{i.e.}, when a higher-dimensional face of the $\ell_1$-ball with center $\bm{x}_F$ and radius $\|\bm{x}_F - \bm{x}_{CF}^*\|_1$ intersects the hyperplane. Suppose that the query returns an optimal $\bm{x}_{CF}^* = \bm{x}_F + \sum_{i=1}^p d_i \bm{e}^i$ with at least two indices $i_1,i_2$ such that $d_{i_1}, d_{i_2} \neq 0$. Then, $
\tilde{\bm{x}}_{CF} := \bm{x}_F + \Big(\sum_{i=1}^p d_i\Big)\bm{e}^{i_1}$ lies on the same face of the $\ell_1$-ball and it has the same $\ell_1$-distance to $\bm{x}_F$. Hence, $\tilde{\bm{x}}_{CF}$ is also an optimal counterfactual and it changes only one coordinate of the factual instance.

For $N_1=\ell_{\infty}$ we can choose $\bm{v} = \sum_{i=1}^p \sgn(a_i)\bm{e}^i$.
Then, $\|\bm{v}\|_{\infty} = 1$ and $\bm{a}^\top\bm{v} = \sum_{i=1}^p |a_i| = \|\bm{a}\|_1=\|\bm{a}\|_{\infty}^*$, so there is an optimal counterfactual obtained by translating $\bm{x}_F$ towards a corner of the $\ell_{\infty}$-ball. Again, when a higher-dimensional face of the $\ell_{\infty}$-ball with center $\bm{x}_F$ and radius $\|\bm{x}_F - \bm{x}_{CF}^*\|_{\infty}$ intersects the hyperplane, the query may return a non-corner optimal point $\bm{x}_{CF}^* = \bm{x}_F + \sum_{i=1}^p d_i \bm{e}^i$ with at least two indices $i_1,i_2$ such that $|d_{i_1}| \neq |d_{i_2}|$. Let $j \in \argmax_{i=1,\ldots,p} |d_i|$, then $\tilde{\bm{x}}_{CF} := \bm{x}_F + d_j \sum_{i=1}^p \sgn(d_i)\bm{e}^i$ lies on the same face of the $\ell_{\infty}$-ball and it has the same $\ell_{\infty}$-distance to $\bm{x}_F$. Thus, $\tilde{\bm{x}}_{CF}$ is an optimal counterfactual obtained by translating $\bm{x}_F$ via a corner point of the $\ell_{\infty}$-ball.

Consequently, for both norms, there always exists an optimal counterfactual that moves the factual instance along a vertex direction of the respective norm ball, and this optimal counterfactual can be reconstructed from any optimal solution.
\end{example}

\begin{lemma}\label{lem:CF_basis}
    Let $\bm{v}$ be a vector from Lemma \ref{lem:CF_direction} and $V=\{\bm{v},\bm{v}^{2},\ldots, \bm{v}^{p}\}$ be a basis of $\mathbb{R}^p$. If $\bm{v}_{CF}$ denotes the optimal counterfactual of $\bm{v}$ and $\bm{v}^{i}_{CF}$ the optimal counterfactual of $\bm{v}^{i}$ as given by \eqref{eq:CF_direction}, then we have the following:
    \begin{enumerate}[(i)]
        \item If $\bm{v}_{CF}\neq\bm{0}$, then the set of counterfactuals $\{\bm{v}_{CF},\bm{v}_{CF}^{2},\ldots, \bm{v}_{CF}^{p}\}$ is linearly independent.
        \item If $\bm{v}_{CF}=\bm{0}$, then the set  of counterfactuals $\{\bm{v}_{CF}^{2},\ldots, \bm{v}_{CF}^{p}\}$ is linearly independent.
    \end{enumerate}
\end{lemma}

Lemmas \ref{lem:CF_direction} and \ref{lem:CF_basis} show that we can always construct a counterfactual of the form $\bm{x}_{CF} = \bm{x}_F + d_{\bm{x}_F}\bm{v}$, and if we have a basis $V$ such that $\bm{v}\in V$, then querying counterfactuals for the whole basis $V$ will lead to a linearly independent set of points that lie on the hyperplane. This allows us to recover all parameters $\bm{a},b$.

Since we do not know $\bm{a},b$ a priori, we also do not know the vector $\bm{v}$ from Lemma \ref{lem:CF_direction}. However, after one counterfactual query for factual $\bm{x}_F$ such that $\bm{x}_F\neq \bm{x}_{CF}$ we can obtain $\bm{v}$ from Lemma \ref{lem:CF_direction} as
\[
\bm{v} = \frac{1}{d_{\bm{x}_F}}\left( \bm{x}_{CF} - \bm{x}_F\right).
\]
Afterwards, we can use the Gram-Schmidt process to create a basis $V$ that contains $\bm{v}$. This procedure is described in \cref{alg:CF_nondiff}.

\begin{algorithm}[H]
\caption{Extracting $\hat{\bm{a}},\hat{b}$ using counterfactuals with non-differentiable norm $\| \cdot \|_{N_1}$}\label{alg:CF_nondiff}
\begin{algorithmic}
\State \textbf{Input:} Counterfactual query $q_{CF}$, norm $f(\bm{x})=\| \bm{x}\|_{N_1}$.
\For{$i=1,\hdots,p$}
\State Query $\bm{x}_{CF}^{(i)*} \leftarrow q_{CF}(\bm{e}^{i})$
\If{$\bm{x}_{CF}^{(i)*}\neq\bm{e}^{i}$}
\State $\bm{v} \gets \bm{x}_{CF}^{(i)*} -\bm{e}^{i}$
\State \textbf{break} \textit{for-loop}
\EndIf
\EndFor
\State Create basis $V$ such that $\bm{v}\in V$ using the Gram-Schmidt process
\State Query the points $\bm{v}_{CF}^{i} \leftarrow q_{CF}(\bm{v}^{i})$ for all $\bm{v}^{i}\in V$
\State $\hat{\bm{a}},b \gets$ find solution to the set of linear equations (\ref{eq:syst_eq_CF})
\State \textbf{return} $\hat{\bm{a}},\hat b$
\end{algorithmic}
\end{algorithm}

\begin{theorem}\label{thm:CF_nondiff} 
Assuming $\bm{x}_{CF}^{(1)*}\neq \bm{e}^{(1)}$, \cref{alg:CF_nondiff} uses $p+1$ counterfactual queries to return parameters $\hat{\bm{a}},\hat b$ of a hyperplane, which is equivalent to the original hyperplane with parameters $\bm{a},b$.
\end{theorem}
Note that the assumption in \cref{thm:CF_nondiff} is not too restrictive since an arbitrary point as $\bm{e}^{(1)}$ is unlikely to lie on the hyperplane. Removing this assumption, it could be possible that all directions $\bm{e}^i$ need to be queried to find $\bm{v}$. However, then we would already have found $p$ independent points on the hyperplane, so we can retrieve $\bm{a},b$. When $\bm{v}$ is found after querying $p-1$ directions, then \cref{alg:CF_nondiff} will query a newly created basis, hence in total $2p-1$ counterfactual queries are needed in the worst case.

\begin{example}    
We demonstrate the effectiveness of our approach using a simple two-dimensional example shown in Figure \ref{fig:CF_Example_2D}. Consider the hyperplane given by $2x_1-x_2=3$, \textit{i.e.}, $\bm{a}=(2,-1), b=3$. We examine a counterfactual mechanism with $N_1=\ell_{\infty}$ as \textit{minimal edit}. First, we ask a counterfactual query for the point $(3,0)$, which yields its optimal counterfactual point $(2,1)$. This gives us the direction of the counterfactuals, $(2,1)-(3,0)=(-1,1)$. Next, we use a counterfactual query for the point $(-1,1)$, which outputs $q_{CF}(-1,1)=(1,-1)$. Note that we have already obtained a linearly independent set of equations to determine the hyperplane parameters:
\begin{align*}
    2a_1+a_2&=b, &
    a_1-a_2&=b.
\end{align*}
By setting $b=1$, we obtain the solution $\bm{a}=\frac{1}{3}(2, 1)$. So, we found a hyperplane that is equivalent to the original hyperplane.
\begin{figure}[h]
    \centering
    \includegraphics[width=0.5\linewidth]{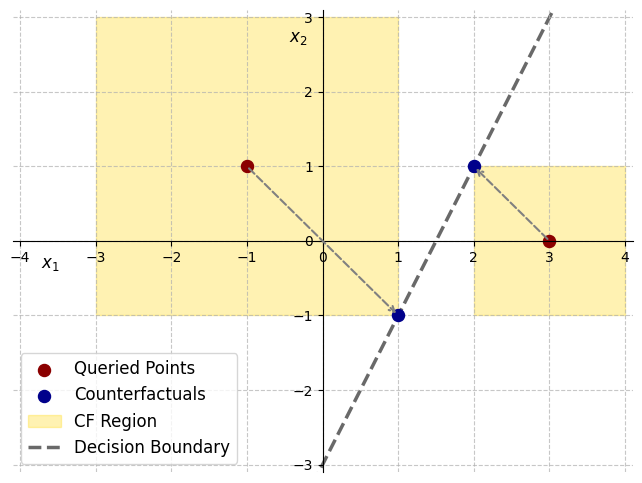}
    \caption{Example of extracting the hyperplane using counterfactual queries with $N_1=\ell_{\infty}$.}
    \label{fig:CF_Example_2D}
\end{figure}
\end{example}

\section{Robust Counterfactual Queries}
Building upon the results of the counterfactual queries, we will first discuss in this section the number of robust counterfactual (RCF) queries needed to extract the original hyperplane. Additionally, we will show how the induced \textit{`Yes'} and \textit{`No'} regions are characterized when not enough queries are given. 

\subsection{Extracting the Hyperplane}
In the following, we will show that the number of RCF queries needed to recover the linear classifier $h_{\bm{a},b}(\bm{x})$ depends on the norm of the robust counterfactual problem \eqref{eq:robust_CF_problem} but is independent of the norm used for the robustness set. The results show that the number of queries needed doubles compared to the classical counterfactuals, which indicates that in our setup, where the robustness radius $\rho$ is known, robust counterfactuals provide an additional level of privacy compared to classical counterfactuals. We first show the following lemma, which is the robust counterpart of Lemma \ref{lem:optimality_conditions_CF}.

\begin{lemma}\label{lem:optimality_conditions_RCF}
Let $\bm{x}_F\in \mathcal X$ be an arbitrary factual instance and $\bm{x}_{RCF}^*\neq \bm{x}_F$ be a corresponding optimal robust counterfactual under an arbitrary norm $f(\bm{x})=\| \bm{x} \|_{N_1}$ and robustness set $\mathcal S=\{ \bm{s}\mid \|\bm{s}\|_{N_2}\le \rho\}$ with $\rho > 0$. Then, there exists a scalar $\lambda^*\in \mathbb R\setminus \{0\}$ such that
\begin{align*}
    & b = \bm{a}^\top \bm{x}^*_{RCF} +q_F(\bm{x}_F) \rho\|\bm{a}\|_{N_2}^*,  \\
    & \lambda^* \bm{a} \in \partial  f(\bm{x}^*_{RCF} - \bm{x}_F) .
\end{align*}
\end{lemma}

As for the classical counterfactuals, the lemma indicates that for differentiable norms, only one query is needed, while the situation with non-differentiable norms is more complicated.

\subsubsection{Differentiable Norms}
For differentiable norms, the subdifferential in Lemma \ref{lem:optimality_conditions_RCF} is a singleton, containing only the gradient of $f$. Hence, in this case the parameters of $\bm{a}$ can directly be extracted. The following theorem shows that we only need one counterfactual query to recover the hyperplane of the classifier.
\begin{theorem}\label{thm:RCF_diff}
Let $q_{RCF}$ be a robust counterfactual mechanism using an arbitrary differentiable norm $f(\bm{x}) = \| \bm{x} \|_{N_1}$ and robustness set $\mathcal S=\{ \bm{s}\mid \| \bm{s}\|_{N_2} \le \rho\}$. One robust counterfactual query and one factual query for an arbitrary data point $\bm{x}_F\in \mathcal X$ are needed to extract the original model's parameters. In particular, for $q_{RCF}(\bm{x}_F)=\bm{x}_{RCF}^*$, let $\hat{\bm{a}} = \nabla f(\bm{x}_{RCF}^*-\bm{x}_F)$ and $\hat b = \hat{\bm{a}}^\top \bm{x}_{RCF}^* +q_F(\bm{x}_F)\rho\| \hat{\bm{a}}\|_{N_2}^*$. Then, it holds that the hyperplane given by $\hat{\bm{a}},\hat{b}$ is equivalent to the original hyperplane with parameters $\bm{a},b$.
\end{theorem}
Note that \cref{thm:CF_theorem_diff} finds an equivalent hyperplane as used by the original classifier, but it could be that the two different classification regions are interchanged. To obtain an equivalent classifier, it suffices to check the outcome of a single factual query, which is already done in \cref{thm:CF_theorem_diff}. Recall that one counterfactual query sufficed for extracting the original model parameters, but an additional factual query is required to determine the classification of the regions. Therefore, extracting the model parameters requires one more query when using robust counterfactuals, while finding an equivalent model requires the same number of queries.

\subsubsection{Non-differentiable Norms}
For non-differentiable norms, the subdifferential set in Lemma \ref{lem:optimality_conditions_RCF} can be larger than a singleton, so \cref{thm:RCF_diff} no longer holds. In order to extract the hyperplane, we follow a similar procedure as for the counterfactual queries. The robust counterpart of Corollary \ref{cor:optimality_conditions_CF_non_differentiable_norm} is derived in the following.
\begin{corollary}\label{cor:optimality_conditions_RCF_non_differentiable_norm}
Let $\bm{x}_F\in \mathcal X$ be an arbitrary factual instance and $\bm{q}_{RCF}$ be a robust counterfactual query under an arbitrary norm $\| \bm{x} \|_{N_1}$ and robustness set $\mathcal S=\{ \bm{s}\mid \|\bm{s}\|_{N_2}\le \rho\}$. Then, $\bm{x}_{RCF}^*\neq \bm{x}_F$ is an optimal robust counterfactual if and only if there exists a $\lambda\in\mathbb R\setminus \{ 0\}$ such that
\begin{align*}
    &\bm{a}^\top \bm{x}_{RCF}^* + q_F(\bm{x}_F)\rho \|\bm{a}\|_{N_2}^* = b, \\
    &  \lambda \bm{a}^\top (\bm{x}_{RCF}^* - \bm{x}_F) = \|\bm{x}_{RCF}^* - \bm{x}_F\|_{N_1}, \\
    & \|\lambda \bm{a} \|_{N_1}^* \le 1.
\end{align*}
\end{corollary}

To retrieve the original hyperplane, we need additional robust counterfactual queries resulting in a system of equalities which has solution $\bm{a},b$. Unlike counterfactuals, robust counterfactuals do not lie on the hyperplane. When we have a set of $p$ linearly independent RCF points $\{\bm{x}^{(1)}_{RCF}, \hdots, \bm{x}^{(p)}_{RCF}\}$, we have the following system of equations
\begin{align}\label{eq:system_eq_RCF_full}
\bm{a}^\top \bm{x}^{(i)}_{RCF} - b+ q_F(\bm{x}^{(i)}_F)\rho\|\bm{a}\|_{N_2}^* & = 0 \quad  \forall i=1,\hdots, p,    
\end{align}
which is nonlinear in $\bm{a},b$ due to the term $\|\bm{a}\|_{N_2}^*$. When we know the classification of the original factual, we know $q_F(\bm{x}^{(i)}_F)$. After scaling we can also assume $\|\bm{a}\|_{N_2}^*=1$ and \eqref{eq:system_eq_RCF_full} can be reformulated as
\begin{align}\label{eq:system_eq_RCF}
\bm{a}^\top \bm{x}^{(i)}_{RCF} - b+ q_F(\bm{x}^{(i)}_F)\rho & = 0 \quad \forall i=1,\hdots, p, \\
\| \bm{a}\|^*_{N_2} & = 1.\notag
\end{align}
The first $p$ equations have a one-dimensional solution space since it is a linear system in dimension $p+1$ with $p$ linearly independent equations. By solving the latter system and afterwards scaling the solution to $\|\bm{a}\|_{N_2}^*=1$, we retrieve a hyperplane equivalent to the original classifier. Note that the latter discussion indicates that for each of the constructed factual points, we need to perform a factual query, which was not needed for the classical counterfactual case.

\begin{remark}
In the case that $\bm{0}$ is a robust counterfactual for an arbitrary point $\bm{x}\in\mathcal{X}$, we obtain by the optimality conditions that $b = - q_F(\bm{x}_F)\rho\|\bm{a}\|_{N_2}^*$. If we know the original classification of $\bm{x}$, the solution space of $\bm{a},b$ boils down to the solution space of $\bm{a}$, which is one-dimensional when we have $p-1$ linearly independent points $\{\bm{x}^{(1)},\ldots, \bm{x}^{(p-1)}\}$ such that
\begin{align*}
\bm{a}^\top \bm{x}^{(i)}-b+ q_F(\bm{x}_F^{(i)})\rho  & = 0 \quad \forall i=1,\ldots, p-1,\\
\| \bm{a}\|^*_{N_2} & = 1.\notag
\end{align*}
Thus, if $\bm{0}$ is a robust counterfactual, then, by finding $p-1$ linearly independent counterfactual vectors, the optimality conditions in Lemma \ref{cor:optimality_conditions_RCF_non_differentiable_norm} still ensure that we retrieve a hyperplane equivalent to the original classifier.    
\end{remark}

The latter discussion indicates that we have to find a set of $p$ linearly independent RCF points. To achieve this, we have to smartly choose the factual instances $\bm{x}_F^{(i)}$. To approach this we show the following lemma.
\begin{lemma}\label{lem:RCF_direction}
    Let $\bm{x}_F\in \mathcal X$ be an arbitrary factual instance and $q_{RCF}(\bm{x}_F) = \bm{x}_{RCF}^*$ under an arbitrary norm $ \| \cdot \|_{N_1}$ and robustness set $\mathcal S=\{ \bm{s}\mid \|\bm{s}\|_{N_2}\le \rho\}$. Then,
    \begin{equation*}\label{eq:RCF_direction}
        \bm{x}_{RCF}^*=\bm{x}_F+d_{\bm{x}_F}\bm{v},
    \end{equation*}
    with $d_{\bm{x}_F} = \frac{b-\bm{a}^{\top}\bm{x}_F- q_F(\bm{x}_F)\rho\|\bm{a}\|_{N_2}^*}{\|\bm{a}\|_{N_1}^*}$ is an optimal robust counterfactual if and only if $\|\bm{v}\|_{N_1}\leq 1$ and $\bm{a}^{\top}\bm{v} =\|\bm{a}\|_{N_1}^*$.
\end{lemma}

The lemma shows, similar to the results for counterfactuals, that there exists a direction $\bm{v}$ such that for any factual point, there exists a robust counterfactual which is the factual point perturbed into this direction $\bm{v}$. Especially, the direction is independent of the factual instance. From the definition of the dual norm it follows that
\[
\|\bm{a}\|^*_{N_1}  = \max_{\| \bm{v}\|_{N_1} \le 1} \bm{a}^\top \bm{v},
\]
and hence, we know that the direction $\bm{v}$ must be a maximizer of the latter problem. Especially, the direction $\bm{v}$ is the same as for classical counterfactuals.

By combining Lemma \ref{lem:CF_basis} with Lemma \ref{lem:RCF_direction}, we can apply the same algorithm used for counterfactuals to recover the hyperplane using robust counterfactuals when norm-1 is non-differentiable. Specifically, in \cref{alg:CF_nondiff}, we replace $q_{CF}$ with both $q_{RCF}$ and $q_F$. Consequently, we need $p+1$ robust counterfactual queries and $p+1$ factual queries to solve the system of equations (\ref{eq:system_eq_RCF}) and obtain hyperplane parameters $\hat{\bm{a}},\hat{b}$ that are equivalent to the original hyperplane parameters $\bm{a},b$.

\begin{corollary}\label{crl:non_diff_RCF_main}
   Let $q_{RCF}$ be a robust counterfactual mechanism using an arbitrary non-differentiable norm $f(\bm{x}) = \|\bm{x}\|_{N_1}$ and robustness set $\mathcal{S} = \{\bm{s} \mid \|\bm{s}\|_{N_2}\leq\rho\}$. Then, with only $p+1$ robust counterfactual queries and $p+1$ factual queries, we can recover hyperplane parameters $\hat{\bm{a}}, \hat{b}$ which are equivalent to the original hyperplane parameters $\bm{a}, b$.
\end{corollary}

\begin{example}
We demonstrate our approach using a simple two-dimensional example shown in \cref{fig:RCF_Example_2D}. Consider the hyperplane given by $2x_1-x_2=3$, \textit{i.e.}, $\bm{a}=(2,-1), b=3$. We examine a counterfactual mechanism with $N_1=\ell_{\infty}$ as \textit{minimal edit} and robustness set given by $\mathcal{S}=\{\bm{s}\mid\|\bm{s}\|_{N_2}\leq 1\}$ with $N_2 = \ell_1$. First, we ask a factual query for $(3,0)$, which outputs \textit{`No'}. A robust counterfactual query for the point $(3,0)$ yields $q_{RCF}(3,0)=\frac{1}{3}(4,5)$. This gives us the same direction as for the counterfactuals, $(-1,1)$. Next,  a factual query $(-1,1)$ yields \textit{`Yes'} and a robust counterfactual query for the point $(-1,1)$ outputs $q_{RCF}(-1,1)=\frac{1}{3}(5,-5)$. Note that we have obtained the following equations:
\[
    \frac{1}{3}(4a_1+5a_2)-b-\|\bm{a}\|_{\infty}=0\ \ \ \text{ and } \ \ \  \frac{1}{3}(5a_1-5a_2)-b+\|\bm{a}\|_{\infty}=0.
\]
By setting $\|\bm{a}\|_{\infty}=1$, we obtain the solution $b=\frac{3}{2} a_1, a_2=\frac{a_1}{10}-\frac{3}{5}$. Moreover, since we enforce $\|\bm{a}\|_{\infty}=1$ we have $\max\{|a_1|,|\frac{a_1}{10}-\frac{3}{5}|\}=1$ implying that $a_1=-1$ or $a_1=1$. For $a_1=-1$, we have the hyperplane $\bm{a}=(-1, -\frac{7}{10}), b= -\frac{3}{2}$. Then, for point $\bm{x}=\frac{1}{3}(4,5)$, we would have $\bm{a}^{\top}\bm{x}-b = -\frac{4}{3}+\frac{3}{2}=\frac{1}{9}>0$ which contradicts with $\bm{x}$ being a robust counterfactual of $(3,0)$. Therefore, we conclude $a_1=1$ giving us the hyperplane $\bm{a}=(1, -\frac{1}{2}), b= \frac{3}{2}$, which is an equivalent hyperplane to the original one.
\begin{figure}[h]
    \centering
    \includegraphics[width=0.7\linewidth]{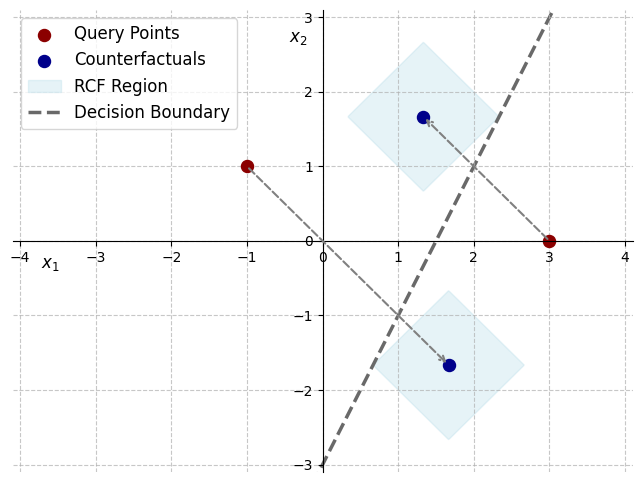}
    \caption{Example of extracting the hyperplane using robust counterfactual queries with $N_1=\ell_{\infty}$ and robustness set given by $\mathcal{S}=\{\bm{s}\mid\|\bm{s}\|_{N_2}\leq 1\}$  with $N_2 = \ell_1$.}
    \label{fig:RCF_Example_2D}
\end{figure}
\end{example}

\subsection{Classification Regions}
In the previous section, we used \cref{thm:RCF_diff} to conclude that if $\| \cdot\|_{N_1}$ is differentiable, then the full hyperplane is extracted using one factual and one counterfactual query. Hence, we exactly know the full classification regions. However, for non-differentiable norms, this is not the case.

Now, suppose that, next to a set of arbitrary factuals, we also have an arbitrary set of robust counterfactuals (RCFs). This means that for points $\bm{x}^{(j)}$ from an index subset $j\in J_{0}\cup J_1\subseteq I_{0}\cup I_1$ we get a RCF $\bm{x}_{RCF}^{(j)}$.   Then, we know that every point $\bm{z}$ in the robust region $\mathcal S$ around the robust counterfactual will be classified differently from the corresponding factual. Concretely,
\begin{align}
  \bm{a}^{\top}\bm{z} -b&\geq 0  \quad\forall \bm{z} : \| \bm{z} - \bm{x}_{RCF}^{(j)}\|_{N_2} \leq \rho&&\quad\forall j \in J_{0},\label{eq:RCF_J0}\\
   \bm{a}^{\top}\bm{z} -b&\leq 0  \quad\forall \bm{z} : \| \bm{z} - \bm{x}_{RCF}^{(j)}\|_{N_2} \leq \rho,\;\;&&\quad\forall j \in J_1.\label{eq:RCF_J1}
\end{align} 
Besides, Lemma \ref{lem:optimality_conditions_RCF} tells us $\bm{a}$ is a subgradient of norm $f(\bm{x}) = \|\bm{x}\|_{N_1}$ at $(\bm{x}_{RCF}^{(j)} - \bm{x}^{(j)})$. That is,
\begin{align}
    \bm{a}\in\partial f&(\bm{x}_{RCF}^{(j)} - \bm{x}^{(j)})\;\;&&\quad\forall j \in J.\label{eq:RCF_sg}
\end{align}

With this extra information, we can characterize the uncertainty set of $\bm{a},b$  as $\mathcal{U}_{\bm{a},b}^{RCF}$ as given by the following set of constraints:
\begin{align*}
  \bm{a}^{\top}\bm{x}^{(i)} -b&\leq 0  &&\quad \forall i \in I_{0}, \notag\\
  \bm{a}^{\top}\bm{x}^{(i)} -b&\geq 0  &&\quad \forall i \in I_1, \notag \\ 
  \bm{a}^{\top}\bm{z} -b&\geq 0  \quad\forall \bm{z} : \| \bm{z} - \bm{x}_{RCF}^{(j)}\|_{N_2} \leq \rho&&\quad\forall j \in J_{0},\\
   \bm{a}^{\top}\bm{z} -b&\leq 0  \quad\forall \bm{z} : \| \bm{z} - \bm{x}_{RCF}^{(j)}\|_{N_2} \leq \rho,\;\;&&\quad\forall j \in J_1,\\
   \bm{a}\in\partial f&(\bm{x}_{RCF}^{(j)} - \bm{x}^{(j)}) \;\;&&\quad\forall j \in J.
\end{align*} 

We remark that for norm-1 equal to $\ell_1$ or $\ell_{\infty}$, constraint (\ref{eq:RCF_sg}) can be linearized as the subgradient of these norms can be modeled using linear constraints. The rest of the uncertainty set is of a similar structure to $\mathcal{U}_{\bm{a},b}^{CF}$ as presented for the counterfactual queries. Therefore, using the same techniques as presented in \cref{sec:regions_CF}, a computationally tractable characterization of the \textit{`Yes'} and \textit{`No'} regions can be derived.

However, we know more information about $\bm{a},b$. We know that the robust region $\mathcal{S}$ will touch the hyperplane since otherwise a closer robust counterfactual would exist. We generally cannot characterize where the robust region around the robust counterfactual touches the hyperplane without knowing the model parameters. This is a key difference with general counterfactuals, where we know that the counterfactual point lies on the hyperplane. Using this information, we can rewrite constraints (\ref{eq:RCF_J0}) and ({\ref{eq:RCF_J1}). To this end, we first rewrite constraint (\ref{eq:RCF_J0}) for the worst case scenario to get an equivalent constraint
\[\bm{a}^{\top}\bm{x}_{RCF}^{(j)} -b-\rho\|\bm{a}\|_{N_2}^*\geq 0 \quad\forall j \in J_{0}.\]
Moreover, since we also know $\mathcal{S}$ touches the hyperplane, equality should hold for this constraint, resulting in the first optimality condition of Lemma \ref{lem:optimality_conditions_RCF}. Similarly, constraint (\ref{eq:RCF_J1}) can be replaced by the following stronger constraint:
\[\bm{a}^{\top}\bm{x}_{RCF}^{(j)} -b+\rho\|\bm{a}\|_{N_2}^*= 0 \quad\forall j \in J_1.\]
Therefore, a stricter uncertainty set of $\bm{a},b$ denoted by $\overline{\mathcal{U}}_{\bm{a},b}^{RCF}$ can be described using the following constraints
\begin{align*}
  \bm{a}^{\top}\bm{x}^{(i)} -b&\leq 0  &&\quad \forall i \in I_{0}, \notag\\
  \bm{a}^{\top}\bm{x}^{(i)} -b&\geq 0  &&\quad \forall i \in I_1, \notag \\ 
  \bm{a}^{\top}\bm{x}_{RCF}^{(j)} -b-\rho\|\bm{a}\|_{N_2}^*&= 0 &&\quad\forall j \in J_0\\
  \bm{a}^{\top}\bm{x}_{RCF}^{(j)} -b+\rho\|\bm{a}\|_{N_2}^*&= 0 &&\quad\forall j \in J_1,\\
   \bm{a}\in\partial  f(\bm{x}_{RCF}^{(j)} -& \bm{x}^{(j)}) \;\;&&\quad\forall j \in J.
\end{align*} 

Unfortunately, these equality constraints in the uncertainty set $\overline{\mathcal{U}}_{\bm{a},b}^{RCF}$ make the inner optimization problems for the \textit{`Yes'} and \textit{`No'} regions,
\begin{align*}
    \mathcal{X}_{\textit{`Yes'}}&=\{ \bm{x}  \;\; | \;\;  \min_{\bm{a},b\in \overline{\mathcal{U}}_{\bm{a},b}^{RCF}}\bm{a}^{\top}\bm{x} - b \geq 0  \}&\text{and}&&
  \mathcal{X}_{\textit{`No'}}&=\{ \bm{x}  \;\; | \;\;  \max_{\bm{a},b\in \overline{\mathcal{U}}_{\bm{a},b}^{RCF}}\bm{a}^{\top}\bm{x} - b \leq 0  \}
\end{align*}
 intractable due to the non-linearity of $\|\cdot\|_{N_2}^*$. In fact, the reformulation in the previous sections relied on a duality argument that cannot be applied in the latter situation.

Using the abovementioned uncertainty set, $\overline{\mathcal{U}}_{\bm{a},b}^{RCF}$, we visualize the classification regions when there is one data point classified as \textit{`No'} and a corresponding robust counterfactual, \textit{i.e.}, $|I|=|J|=1$. We examined nine cases, combining norm-1 and norm-2 pairs from $\{\ell_1,\ell_2,\ell_{\infty}\}$. The results are depicted in Figure \ref{fig:RCF_norms}, where we see that for norm-1 non-differentiable, \textit{i.e.}, $\ell_1,\ell_{\infty}$, there are areas where we cannot conclude the classification.

Two notable insights result in conditions on $\bm{a},b$, which can be added to $\mathcal{U}_{\bm{a},b}^{RCF}$ such that the dualization for the inner optimization problems for the \textit{`Yes'} and \textit{`No'} region remains tractable.

First, we see in Figure \ref{fig:RCF_norms} that for norm-1 non-differentiable, the \textit{`No'} region is a translated pointed cone. The vertex point of this cone is not always equal to the original factual queried data point. In the following, we will characterize this point $\bar{\bm{x}}$. This results in an extra data point for which we know the classification, which reduces the uncertainty in $\bm{a},b$.

\begin{lemma}\label{lem:perspective_point}
Let $\bm{x}_F$ be a data point and let $\bm{x}_{RCF}^*$ be an optimal robust counterfactual under $\| \cdot\|_{N_1}$-distance and with robustness set $\mathcal{S}=\{\bm{s}\mid \|\bm{s}\|_{N_2}\leq \rho\}$. Furthermore, assume that $\|\bm{a}\|_{N_2}^*\le C \| \bm{a} \|_{N_1}^*$ for all $\bm{a}\in \mathbb R^p$. Then, for $\bm{v}=(\bm{x}_{RCF}^*-\bm{x}_F)/\|\bm{x}_{RCF}^*-\bm{x}_F\|_{N_1}$ the point $\bar{\bm{x}}:= \bm{x}_{RCF} -  d\bm{v}$ has the same classification as $\bm{x}_F$ if $d\ge \rho C$.
\end{lemma}
We can conclude from the latter lemma the following special cases.
\begin{corollary}\label{cor:N1N2}
Under the same assumptions as in Lemma \ref{lem:perspective_point} the following results hold:
    \begin{enumerate}[(i)]
        \item For $N_1=\ell_1$ and $N_2=\ell_q$ with $q\ge 1$, the point $\bar{\bm{x}}$ has the same classification as $\bm{x}_F$ when $d\geq \rho p^{1-1/q}$.
        \item For $N_1=\ell_{\infty}$ and $N_2=\ell_q$ with $q\ge 1$, the point $\bar{\bm{x}}$ has the same classification as $\bm{x}_F$ when $d\geq \rho$.
        \item If $N_1=N_2$, then $\bm{x}_\mathcal{S}=\bm{x}_{RCF}^*-\rho \bm{v}$ lies on the hyperplane.
    \end{enumerate}
\end{corollary}

The latter corollary shows that in the case when $N_1=N_2$, see, \textit{e.g.}, Figure \ref{fig:ex_rgns_RCF_L1L1} and Figure \ref{fig:ex_rgns_RCF_LinfLinf}, it is possible to characterize an exact point $\bm{x}_\mathcal{S}$, where the robust region $\mathcal{S}$ around the robust counterfactual touches the hyperplane. Using this point, constraints (\ref{eq:RCF_J0}) and (\ref{eq:RCF_J1}) can be reformulated to inequalities stating the robustness set to be on the desired side of the hyperplane, while adding constraint $\bm{a}^{\top}\bm{x}_\mathcal{S}-b=0$.

\begin{figure}[H]
     \centering
     \begin{subfigure}[b]{0.32\textwidth}
         \centering
         \includegraphics[width=\textwidth]{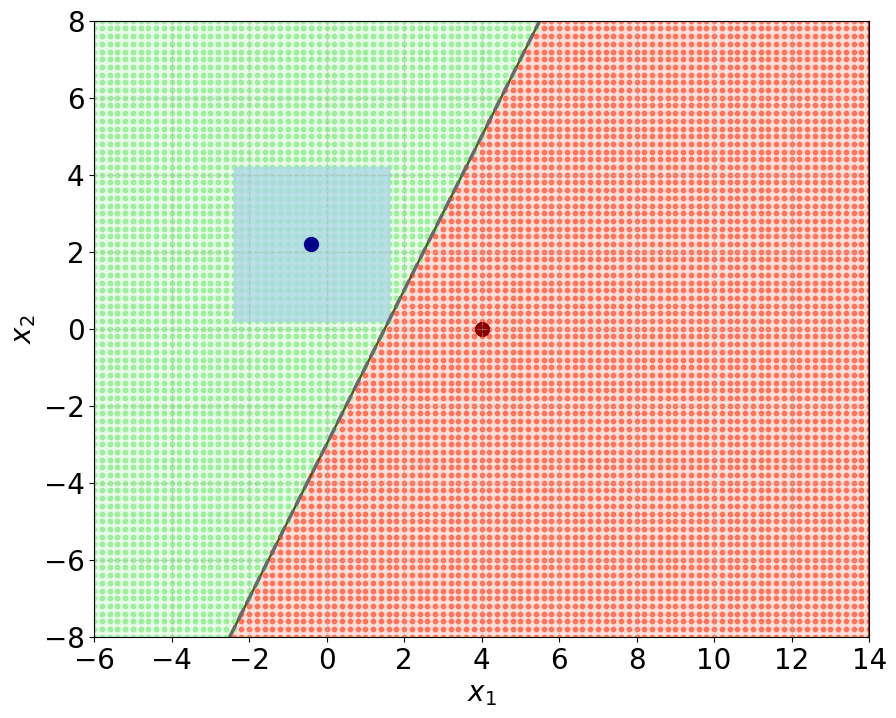}
         \caption{$N_1=\ell_2, N_2=\ell_{\infty}$.}
         \label{fig:ex_rgns_RCF_L2Linf}
     \end{subfigure}
     \hfill
     \begin{subfigure}[b]{0.32\textwidth}
         \centering
         \includegraphics[width=\textwidth]{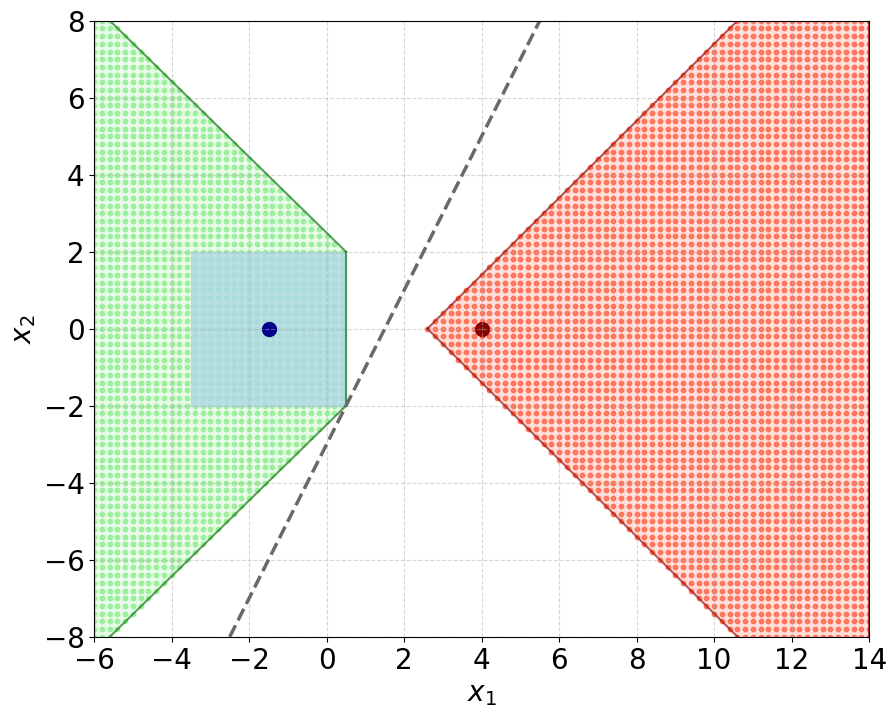}
         \caption{$N_1=\ell_1, N_2=\ell_{\infty}$.}
         \label{fig:ex_rgns_RCF_L1Linf}
     \end{subfigure}
     \hfill
     \begin{subfigure}[b]{0.32\textwidth}
         \centering
         \includegraphics[width=\textwidth]{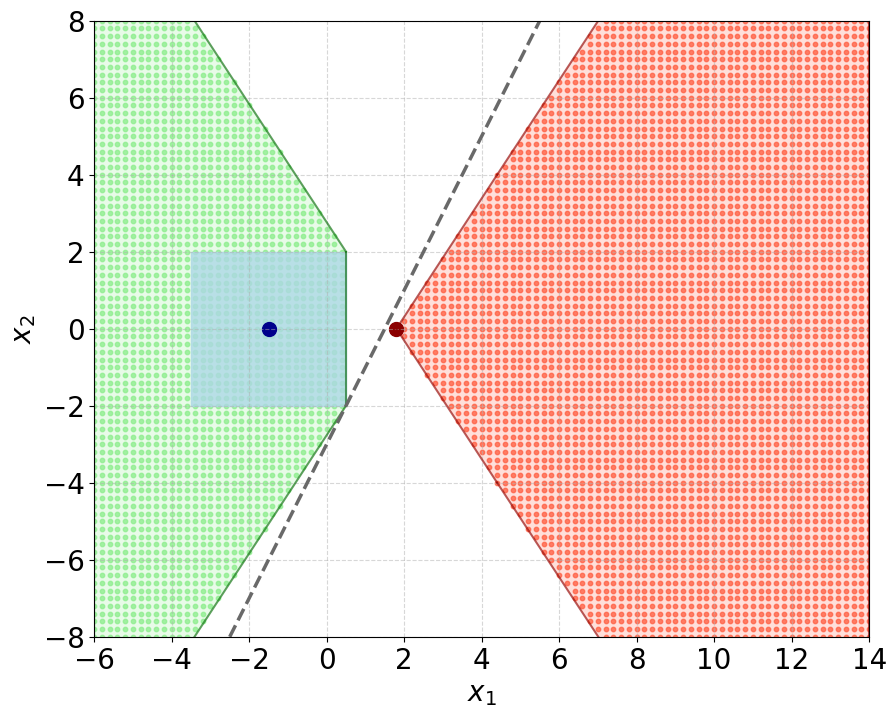}
         \caption{$N_1=\ell_1, N_2=\ell_{\infty}$.}
         \label{fig:ex_rgns_RCF_L1Linf_othr}
     \end{subfigure}
     \hfill
     \begin{subfigure}[b]{0.32\textwidth}
         \centering
         \includegraphics[width=\textwidth]{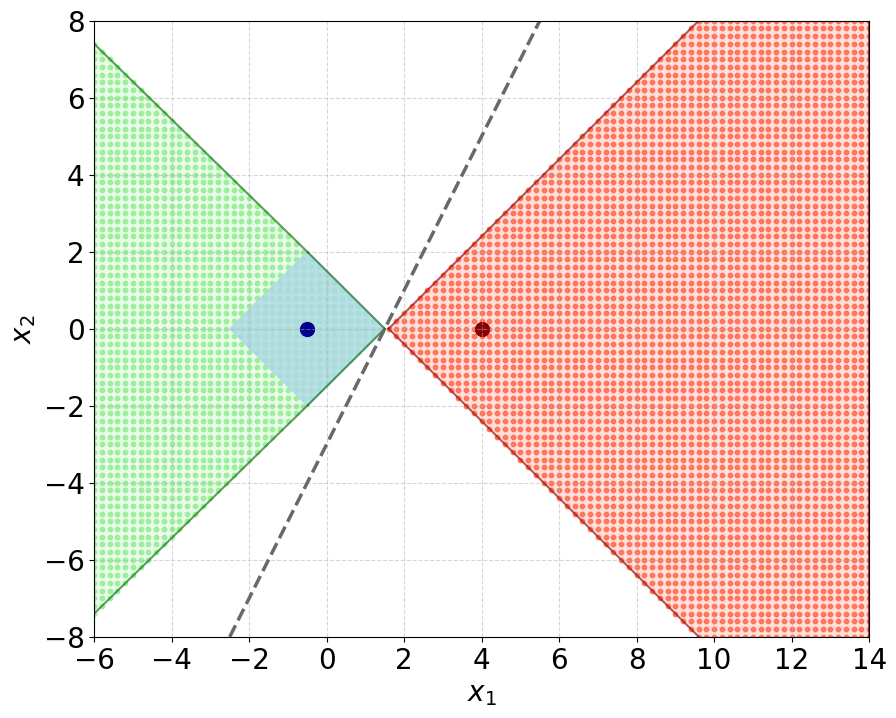}
         \caption{$N_1=\ell_1, N_2=\ell_{1}$.}
         \label{fig:ex_rgns_RCF_L1L1}
     \end{subfigure}
     \hfill
     \begin{subfigure}[b]{0.32\textwidth}
         \centering
         \includegraphics[width=\textwidth]{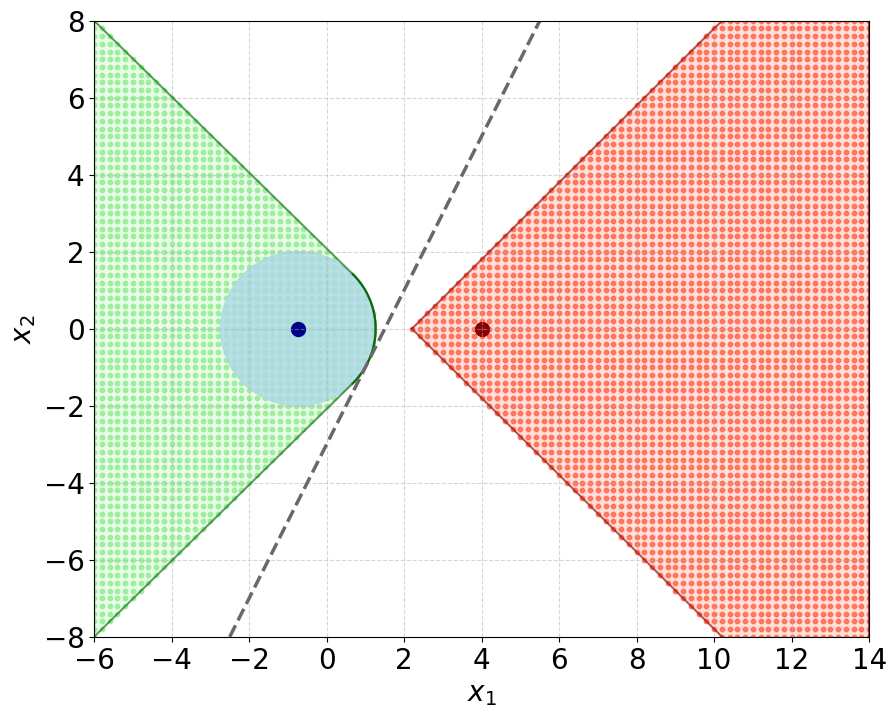}
         \caption{$N_1=\ell_1, N_2=\ell_{2}$.}
         \label{fig:ex_rgns_RCF_L1L2}
     \end{subfigure}
     \hfill
     \begin{subfigure}[b]{0.32\textwidth}
         \centering
         \includegraphics[width=\textwidth]{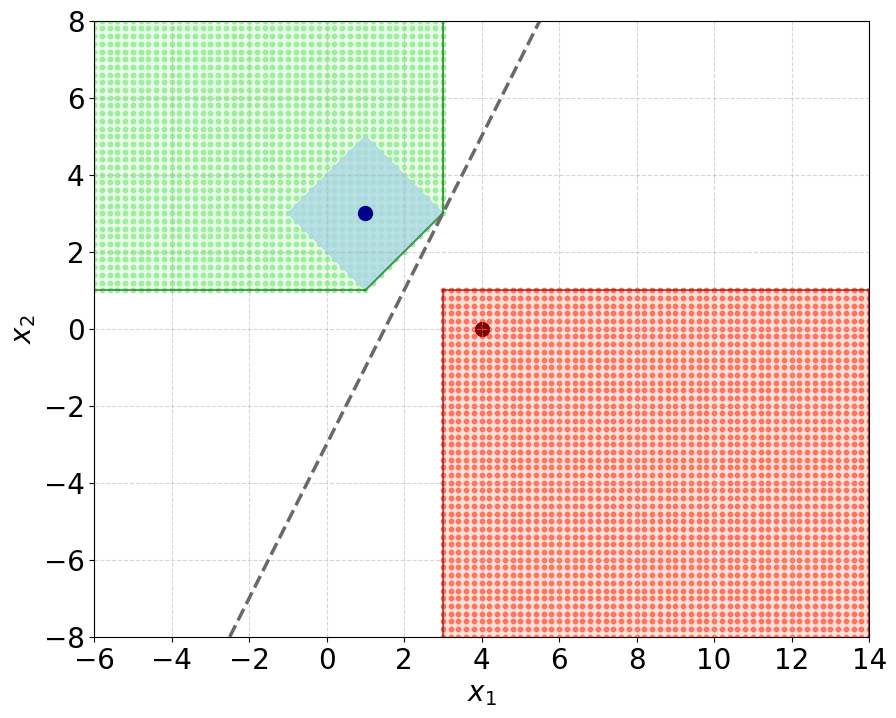}
         \caption{$N_1=\ell_{\infty}, N_2=\ell_{1}$.}
         \label{fig:ex_rgns_RCF_LinfL1}
     \end{subfigure}
     \hfill
     \begin{subfigure}[b]{0.32\textwidth}
         \centering
         \includegraphics[width=\textwidth]{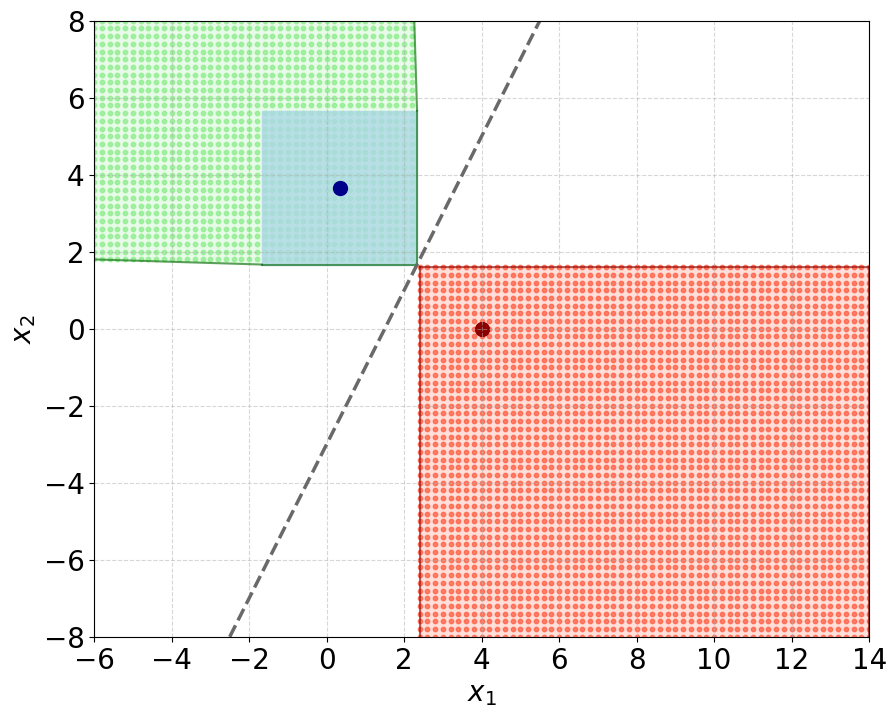}
         \caption{$N_1=\ell_{\infty}, N_2=\ell_{\infty}$.}
         \label{fig:ex_rgns_RCF_LinfLinf}
     \end{subfigure}
     \hfill
     \begin{subfigure}[b]{0.32\textwidth}
         \centering
         \includegraphics[width=\textwidth]{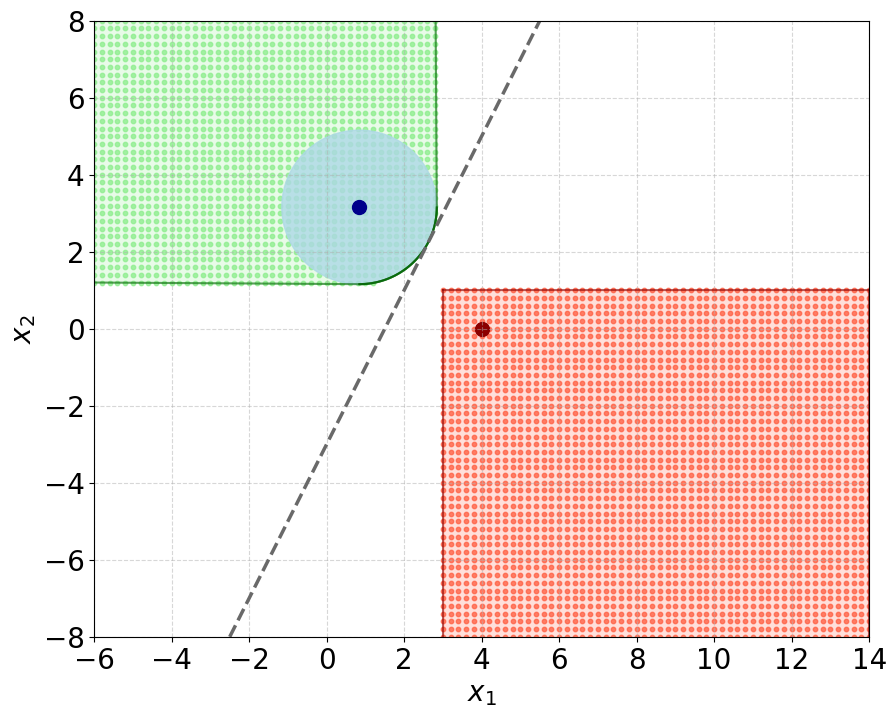}
         \caption{$N_1=\ell_{\infty}, N_2=\ell_{2}$.}
         \label{fig:ex_rgns_RCF_LinfL2}
     \end{subfigure}
     \hfill
     \begin{subfigure}[b]{0.32\textwidth}
         \centering
         \includegraphics[width=\textwidth]{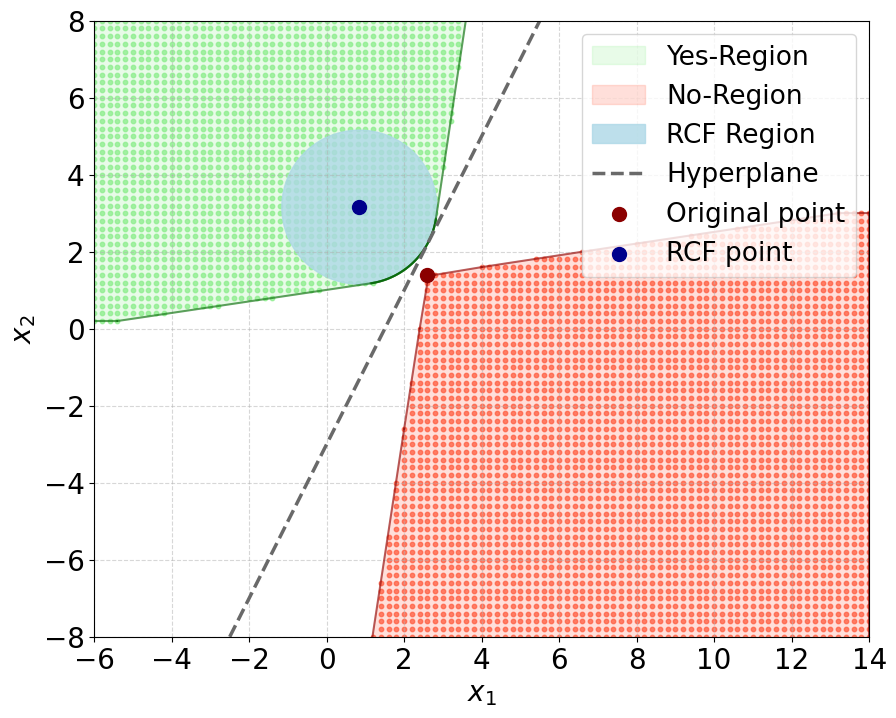}
         \caption{$N_1=\ell_{\infty}, N_2=\ell_{2}$.}
         \label{fig:ex_rgns_RCF_LinfL2_othr}
     \end{subfigure}
     \caption{Examples of the classification regions given one \textit{`No'} factual and corresponding robust counterfactual for different norms.}
        \label{fig:RCF_norms}
\end{figure}

\section{Conclusion}
In this work, we presented how factual, counterfactual, and robust counterfactual queries can be used to extract information of linear classifiers in $p$-dimensional feature space. We provided computationally tractable characterizations of the \textit{`Yes'} and \textit{`No'} regions given an arbitrary set of factual, counterfactual, or robust counterfactual queried points. This allows us to efficiently check which points will be classified as \textit{`Yes'} and which as \textit{`No'} without querying the model again. For factual queries, we have demonstrated that the classification regions extend beyond the convex hull of the factual data points. When we additionally consider counterfactuals, we gain more information about the classification regions. In particular, we have proven that only one counterfactual can extract the full hyperplane when the \textit{minimal edit} distance function is differentiable. When this distance function is non-differentiable, $p+1$ counterfactual queries suffice to retrieve the full hyperplane. We have seen that these results can be generalized for robust counterfactuals. We show that deriving tractable formulations for the \textit{`Yes'} and \textit{`No'} regions is more difficult than for classical counterfactuals but can be done for certain special cases. We have also shown that one pair of factual and robust counterfactual queries suffices to extract the original hyperplane when the \textit{minimal edit} distance function is differentiable. When this distance function is non-differentiable, $p+1$ pairs of factual and robust counterfactual queries are needed. Notably, the norm used to define the robust region does not affect the number of queries needed to extract the model, but can affect the information gained on the classification regions. In summary, our results show that to increase privacy of a linear model, using non-differentiable norms for the distance of the counterfactuals is beneficial. Furthermore, providing robust counterfactuals adds an additional layer of privacy, since for recovery of the model parameters additional factual queries are needed.

Our results naturally come with limitations. Since we assume unconstrained factual and counterfactual queries in $\mathbb{R}^p$, the results presented in this paper are not evidently generalizable to non-continuous data, \textit{e.g.}, categorical or binary. In practice, some features may be immutable, leading to constrained counterfactuals, which we do not address. Future work includes extending our approach to linear models with kernels and to other non-linear models, such as classifiers using quadratic polynomials.
It could also be applied to extract classification regions and parameters of decision tree models, which rely on separating hyperplanes. Furthermore, our work relies on the assumption that the queried (robust) counterfactuals are exact, \textit{i.e.}, optimal solutions of the corresponding optimization problems. In practice, often heuristic methods are used, returning non-optimal (robust) counterfactuals. Extending our work for this situation could be of practical interest.  Finally, another direction is to develop defense mechanisms against the model-extraction techniques proposed in this work.

\bibliographystyle{plainnat}   
\bibliography{references}

\begin{thebibliography}{26}
\providecommand{\natexlab}[1]{#1}
\providecommand{\url}[1]{\texttt{#1}}
\expandafter\ifx\csname urlstyle\endcsname\relax
  \providecommand{\doi}[1]{doi: #1}\else
  \providecommand{\doi}{doi: \begingroup \urlstyle{rm}\Url}\fi

\bibitem[A{\"\i}vodji et~al.(2020)A{\"\i}vodji, Bolot, and
  Gambs]{aivodji2020model}
Ulrich A{\"\i}vodji, Alexandre Bolot, and S{\'e}bastien Gambs.
\newblock Model extraction from counterfactual explanations.
\newblock \emph{arXiv preprint arXiv:2009.01884}, 2020.

\bibitem[{Basel Committee on Banking Supervision}(2026)]{BCBS}
{Basel Committee on Banking Supervision}.
\newblock The basel committee -- overview, 2026.
\newblock URL \url{https://www.bis.org/bcbs/index.htm}.
\newblock Last access: 1 February 2026.

\bibitem[Berning et~al.(2024)Berning, Dunning, Spagnuelo, Veugen, and Van
  Der~Waa]{berning2024trade}
Sjoerd Berning, Vincent Dunning, Dayana Spagnuelo, Thijs Veugen, and Jasper Van
  Der~Waa.
\newblock The trade-off between privacy \& quality for counterfactual
  explanations.
\newblock In \emph{Proceedings of the 19th International Conference on
  Availability, Reliability and Security}, pages 1--9, 2024.

\bibitem[Bertsimas and den Hertog(2022)]{bertsimas_robust_2022}
Dimitris Bertsimas and Dick den Hertog.
\newblock \emph{Robust and {Adaptive} {Optimization}}.
\newblock Dynamic Ideas LLC, 2022.
\newblock ISBN 978-1-7337885-2-6.

\bibitem[Boenisch et~al.(2023)Boenisch, Dziedzic, Schuster, Shamsabadi,
  Shumailov, and Papernot]{boenisch2023curious}
Franziska Boenisch, Adam Dziedzic, Roei Schuster, Ali~Shahin Shamsabadi, Ilia
  Shumailov, and Nicolas Papernot.
\newblock When the curious abandon honesty: Federated learning is not private.
\newblock In \emph{2023 IEEE 8th European Symposium on Security and Privacy
  (EuroS\&P)}, pages 175--199. IEEE, 2023.

\bibitem[Chandrasekaran et~al.(2020)Chandrasekaran, Chaudhuri, Giacomelli, Jha,
  and Yan]{chandrasekaran2020exploring}
Varun Chandrasekaran, Kamalika Chaudhuri, Irene Giacomelli, Somesh Jha, and
  Songbai Yan.
\newblock Exploring connections between active learning and model extraction.
\newblock In \emph{29th USENIX Security Symposium (USENIX Security 20)}, pages
  1309--1326, 2020.

\bibitem[Dissanayake and Dutta(2024)]{dissanayake2024model}
Pasan Dissanayake and Sanghamitra Dutta.
\newblock Model reconstruction using counterfactual explanations: A perspective
  from polytope theory.
\newblock \emph{Advances in Neural Information Processing Systems},
  37:\penalty0 83397--83429, 2024.

\bibitem[{European Union}(2016)]{GDPR}
{European Union}.
\newblock Art. 22 {GDPR} -- {A}utomated individual decision-making, including
  profiling, 2016.
\newblock URL
  \url{https://gdpr.eu/article-22-automated-individual-decision-making/}.
\newblock Last access: 1 February 2026.

\bibitem[{Federal Reserve System}(2011)]{Fed117}
{Federal Reserve System}.
\newblock Sr 11-7: Guidance on {M}odel {R}isk {M}anagement.
\newblock Technical report, Board of Governors of the Federal Reserve System,
  Division of Banking Supervision and Regulation, August 2011.
\newblock URL
  \url{https://www.federalreserve.gov/supervisionreg/srletters/sr1107.htm}.
\newblock Last access: 1 February 2026.

\bibitem[Ferry et~al.(2024)Ferry, Fukasawa, Pascal, and
  Vidal]{ferry2024trained}
Julien Ferry, Ricardo Fukasawa, Timoth{\'e}e Pascal, and Thibaut Vidal.
\newblock Trained random forests completely reveal your dataset.
\newblock \emph{arXiv preprint arXiv:2402.19232}, 2024.

\bibitem[Fredrikson et~al.(2015)Fredrikson, Jha, and
  Ristenpart]{fredrikson2015model}
Matt Fredrikson, Somesh Jha, and Thomas Ristenpart.
\newblock Model inversion attacks that exploit confidence information and basic
  countermeasures.
\newblock In \emph{Proceedings of the 22nd ACM SIGSAC conference on computer
  and communications security}, pages 1322--1333, 2015.

\bibitem[Goethals et~al.(2023)Goethals, S{\"o}rensen, and
  Martens]{goethals2023privacy}
Sofie Goethals, Kenneth S{\"o}rensen, and David Martens.
\newblock The privacy issue of counterfactual explanations: explanation linkage
  attacks.
\newblock \emph{ACM Transactions on Intelligent Systems and Technology},
  14\penalty0 (5):\penalty0 1--24, 2023.

\bibitem[Khouna et~al.(2025)Khouna, Ferry, and
  Vidal]{khouna2025counterfactuals}
Awa Khouna, Julien Ferry, and Thibaut Vidal.
\newblock From counterfactuals to trees: Competitive analysis of model
  extraction attacks.
\newblock \emph{arXiv preprint arXiv:2502.05325}, 2025.

\bibitem[Lowd and Meek(2005)]{lowd2005adversarial}
Daniel Lowd and Christopher Meek.
\newblock Adversarial learning.
\newblock In \emph{Proceedings of the eleventh ACM SIGKDD international
  conference on Knowledge discovery in data mining}, pages 641--647, 2005.

\bibitem[Maragno et~al.(2024)Maragno, Kurtz, R{\"o}ber, Goedhart, Birbil, and
  den Hertog]{maragno2024finding}
Donato Maragno, Jannis Kurtz, Tabea~E R{\"o}ber, Rob Goedhart, {\c{S}}~Ilker
  Birbil, and Dick den Hertog.
\newblock Finding regions of counterfactual explanations via robust
  optimization.
\newblock \emph{INFORMS Journal on Computing}, 36\penalty0 (5):\penalty0
  1316--1334, 2024.

\bibitem[Milli et~al.(2019)Milli, Schmidt, Dragan, and Hardt]{milli2019model}
Smitha Milli, Ludwig Schmidt, Anca~D Dragan, and Moritz Hardt.
\newblock Model reconstruction from model explanations.
\newblock In \emph{Proceedings of the Conference on Fairness, Accountability,
  and Transparency}, pages 1--9, 2019.

\bibitem[Nguyen et~al.(2024)Nguyen, Huynh, Ren, Nguyen, Nguyen, Yin, and
  Nguyen]{nguyen2024survey}
Thanh~Tam Nguyen, Thanh~Trung Huynh, Zhao Ren, Thanh~Toan Nguyen, Phi~Le
  Nguyen, Hongzhi Yin, and Quoc Viet~Hung Nguyen.
\newblock A survey of privacy-preserving model explanations: Privacy risks,
  attacks, and countermeasures.
\newblock \emph{arXiv preprint arXiv:2404.00673}, 2024.

\bibitem[Oksuz et~al.(2024)Oksuz, Halimi, and Ayday]{Oksuz_2024}
Abdullah~Caglar Oksuz, Anisa Halimi, and Erman Ayday.
\newblock Autolycus: Exploiting explainable artificial intelligence (xai) for
  model extraction attacks against interpretable models.
\newblock \emph{Proceedings on Privacy Enhancing Technologies}, 2024\penalty0
  (4):\penalty0 684–699, October 2024.
\newblock ISSN 2299-0984.
\newblock \doi{10.56553/popets-2024-0137}.
\newblock URL \url{http://dx.doi.org/10.56553/popets-2024-0137}.

\bibitem[Pal et~al.(2020)Pal, Gupta, Shukla, Kanade, Shevade, and
  Ganapathy]{pal2020activethief}
Soham Pal, Yash Gupta, Aditya Shukla, Aditya Kanade, Shirish Shevade, and Vinod
  Ganapathy.
\newblock Activethief: Model extraction using active learning and unannotated
  public data.
\newblock In \emph{Proceedings of the AAAI conference on artificial
  intelligence}, volume~34, pages 865--872, 2020.

\bibitem[Reith et~al.(2019)Reith, Schneider, and
  Tkachenko]{reith2019efficiently}
Robert~Nikolai Reith, Thomas Schneider, and Oleksandr Tkachenko.
\newblock Efficiently stealing your machine learning models.
\newblock In \emph{Proceedings of the 18th ACM Workshop on Privacy in the
  Electronic Society}, pages 198--210, 2019.

\bibitem[Rigaki and Garcia(2023)]{rigaki2023survey}
Maria Rigaki and Sebastian Garcia.
\newblock A survey of privacy attacks in machine learning.
\newblock \emph{ACM Computing Surveys}, 56\penalty0 (4):\penalty0 1--34, 2023.

\bibitem[Rockafellar(1997)]{rockafellar1997convex}
R~Tyrrell Rockafellar.
\newblock \emph{Convex {Analysis}}, volume~28.
\newblock Princeton University Press, 1997.

\bibitem[Shokri et~al.(2021)Shokri, Strobel, and Zick]{shokri2021privacy}
Reza Shokri, Martin Strobel, and Yair Zick.
\newblock On the privacy risks of model explanations.
\newblock In \emph{Proceedings of the 2021 AAAI/ACM Conference on AI, Ethics,
  and Society}, pages 231--241, 2021.

\bibitem[Tram{\`e}r et~al.(2016)Tram{\`e}r, Zhang, Juels, Reiter, and
  Ristenpart]{tramer2016stealing}
Florian Tram{\`e}r, Fan Zhang, Ari Juels, Michael~K Reiter, and Thomas
  Ristenpart.
\newblock Stealing machine learning models via prediction $\{$APIs$\}$.
\newblock In \emph{25th USENIX security symposium (USENIX Security 16)}, pages
  601--618, 2016.

\bibitem[Wachter et~al.(2017)Wachter, Mittelstadt, and
  Russell]{wachter2017counterfactual}
Sandra Wachter, Brent Mittelstadt, and Chris Russell.
\newblock Counterfactual explanations without opening the black box: Automated
  decisions and the gdpr.
\newblock \emph{Harvard Journal of Law \& Technology}, 31:\penalty0 841, 2017.

\bibitem[Wang et~al.(2022)Wang, Qian, and Miao]{wang2022dualcf}
Yongjie Wang, Hangwei Qian, and Chunyan Miao.
\newblock Dualcf: Efficient model extraction attack from counterfactual
  explanations.
\newblock In \emph{Proceedings of the 2022 ACM Conference on Fairness,
  Accountability, and Transparency}, pages 1318--1329, 2022.

\end{thebibliography}

\newpage

\appendix

\section{Proofs}\label{app:proofs}
\subsection{Factuals}\label{app:proofs_factuals}
\begin{proof}[\cref{thm:rgn_factuals}]
    When we examine the inner optimization problem of \eqref{eq:no-rgn-opt}, we can use dualization with respect to $\bm{a},b$ to get an equivalent system of equations:
\[\begin{array}{rlcrllc}
     \max&\bm{a}^{\top}\bm{x} - b &\iff& \max &\;(\bm{x},-1)^{\top}\bm{a},b\\
    s.t. &\bm{a},b\in \mathcal{U}_{\bm{a},b}^F.&&s.t. &\;(\bm{x}^{(i)},-1)^{\top}\bm{a},b\leq 0 &\forall i\in I_{0},\\
    && &&\;(-\bm{x}^{(i)},1)^{\top}\bm{a},b\leq 0 &\forall i\in I_i.\\
    &&\iff& \min &\;\bm{0}\\
    &&&s.t. &\sum_{i\in I_{0}}u_i\bm{x}^{(i)} - \sum_{i\in I_1}u_i\bm{x}^{(i)}=\bm{x},\\
    &&&&-\sum_{i\in I_{0}}u_i + \sum_{i\in I_1}u_i=-1,\\
    &&&& \bm{u}\geq \bm{0}.
\end{array}
\]
Hence, the \textit{`No'} region can be written as the following polyhedral set:
\[ \mathcal{X}_{\textit{`No'}} = \left\{ \bm{x} \; | \;\; \exists \bm{u} \; : \;
\sum_{i\in I_{0}} u_i- \sum_{i\in I_1} u_i = 1, \; 
\sum_{i\in I_{0}} \bm{x}^{(i)}u_i- \sum_{i\in I_1} \bm{x}^{(i)}u_i = \bm{x}, \; 
 \bm{u}\geq \bm{0}\right\}. \]
Similarly, using dualization the \textit{`Yes'} region can be described as
\[ \mathcal{X}_{\textit{`Yes'}} =\left\{ \bm{x} \; | \;\; \exists \bm{u} \; : \;
\sum_{i\in I_1} u_i- \sum_{i\in I_{0}} u_i = 1, \; 
\sum_{i\in I_1} \bm{x}^{(i)}u_i- \sum_{i\in I_{0}} \bm{x}^{(i)}u_i = \bm{x}, \; 
 \bm{u}\geq \bm{0}\right\}. \]
\end{proof}

\subsection{Counterfactuals}\label{app:proofs_counterfactuals}
\begin{proof}[\cref{thm:region-CF}]
To characterize the \textit{`No'} region, we need to check when the inner optimization problem as described in equation \ref{eq:no-rgn-opt} is at most 0. Using the uncertainty set of $\bm{a},b$ this optimization problem becomes
$$\begin{array}{lrlcl}
\underset{\bm{a},b}{\max}& \;\bm{a},b^{\top}(\bm{x},-1)\\
 s.t.&\;  \bm{a},b^{\top}(\bm{x}^{(i)},-1)&\leq 0 & &\quad\forall i \in I_{0},  \\
&\;  \bm{a},b^{\top}(-\bm{x}^{(i)},1) &\leq 0  &\quad&\quad\forall i \in I_1, \\
&\;  \bm{a},b^{\top}\bm{z}&\leq 0  &\quad \forall \bm{z} \in\mathcal{U}_j &\quad \forall j \in J_{0}, \\ 
&\;  \bm{a},b^{\top}\bm{z} &\leq 0 &\quad
  \forall \bm{z} \in\mathcal{U}_j&\quad \forall j \in J_1, \\
&\;  \bm{a},b^{\top}(\bm{x}_{CF}^{(j)}, -1) &= 0  &\quad&\quad\forall j \in J,
\end{array}$$
where the uncertainty sets for $j\in J_{0}$ are given by 
\[
\mathcal{U}_j=\{\bm{z}:=(\bm{z}_a,z_b)\mid\| \bm{z}_a - \bm{x}^{(j)}\|_{N_1} \leq \rho_j,\; z_b=-1\}\]
and for $j\in J_1$ by
\[\mathcal{U}_j=\{\bm{z}:=(\bm{z}_a,z_b)\mid \|\bm{z}_a + \bm{x}^{(j)}\|_{N_1} \leq \rho_j,\; z_b=1\}.
\]
We clearly see that for each $j\in J$ the uncertainty set can be written as 
$\mathcal{U}_j=\{\bm{z}\mid f_{jk}(\bm{z})\leq 0, \forall k\in K_j\}$, where $f_{jk}$ is a convex function and $K_j$ the set of indices of constraints that defines uncertainty set $\mathcal{U}_j$. Besides, we see that the uncertainty sets are bounded.
Using the dualization results, primal worst is dual best, as presented in \citet[equation (2.45)]{bertsimas_robust_2022}, we can conclude this optimization problem is equivalent to:
$$\begin{array}{lllll}
\underset{\bm{t,u,v,z}}{\min}& \;0\\
 s.t.&\;  \sum\limits_{i \in I_{0}}t_{i}(\bm{x}^{(i)},-1)+\sum\limits_{i \in I_1}t_{i}(-\bm{x}^{(i)},1) +\sum\limits_{j\in J}v_j (\bm{x}_{CF}^{(j)},-1)+\sum\limits_{j\in J}u_j\bm{z}^{(j)}&=(\bm{x},-1),\\
 &\;\bm{z}^{(j)}\in\mathcal{U}_j\;\;\forall j\in J,\\
 &\; \bm{t},\bm{u}\in\mathbb{R}^{|I|}_{\geq 0},\; \bm{v}\in\mathbb{R}^{|J|},\; \bm{z}^{(j)}\in\mathbb{R}^{p+1}\;\;\forall j\in J. &&
\end{array}$$
Let $\bm{y}^{j}= (\bm{y}^{j}_a,y_b^j)=u_j\bm{z}^{(j)}$. Then, following \citet[Theorem 2.2]{bertsimas_robust_2022}, we write it as the following conic quadratic problem:
$$\begin{array}{lllll}
\underset{\bm{t,u,v,y}}{\min}& \;0\\
 s.t.&\;  \sum_{i \in I_{0}}t_{i}(\bm{x}^{(i)},-1)+\sum_{i \in I_1}t_{i}(-\bm{x}^{(i)},1), \\
 &\;+\sum_{j\in J}v_j (\bm{x}_{CF}^{(j)},-1)+\sum_{j\in J}(\bm{y}_a^{(j)},y_b^{(j)})=(\bm{x},-1), &\\
 &\;u_j(\| \bm{y}_a^{(j)}/u_j - \bm{x}^{(j)}\|_{N_1} - \rho_j)\leq  0&\;\;\forall j\in J_{0}, \\
 &\;u_j (y_b^{(j)}/u_j + 1)=0&\;\;\forall j\in J_{0},\\
 &\;u_j(\| \bm{y}_a^{(j)}/u_j + \bm{x}^{(j)}\|_{N_1} - \rho_j)\leq  0&\;\;\forall j\in J_1,\\
 &\;u_j (y_b^{(j)}/u_j - 1)=0&\;\;\forall j\in J_1,\\
 &\; \bm{t}\in\mathbb{R}^{|I|}_{\geq 0},\; \bm{u}\in\mathbb{R}^{|J|}_{\geq 0},\; \bm{v}\in\mathbb{R}^{|J|},\; \bm{y}^{(j)}\in\mathbb{R}^{p+1}\;\;\forall j\in J.&&
\end{array}$$
Since the objective of the latter optimization problem is constant 0, finding out whether $\bm{x}$ will be classified as a \textit{`No'} can be determined by checking if there exist $\bm{t},\bm{u},\bm{v},\bm{y}$ that meet the constraints.
The proof for the \textit{`Yes'} region follows a similar argumentation.
\end{proof}

\begin{proof}[Lemma \ref{lem:optimality_conditions_CF}]
Without loss of generality, we may assume that $\bm{x}_F$ is classified as \textit{`No'}. Given the classifier $h_{\bm{a},b}$, a potential counterfactual point must lie in the half-space $\bm{a}^\top \bm{x} \ge b$. Since we are minimizing the distance to $\bm{x}_F$, an optimal counterfactual actually lies on the boundary of the half-space $\bm{a}^\top \bm{x} = b$. Hence, the optimization problem \eqref{eq:CF_problem} can be reformulated as
\begin{equation*}\label{eq:CF_problem_reform}
    \begin{aligned}
    \min_{\bm{x}_{CF}} & \ \| \bm{x}_{CF}-\bm{x}_F\|_{N_1} \\
    s.t. \quad & \bm{a}^\top \bm{x}_{CF} = b .
\end{aligned}
\end{equation*}
The latter problem has a convex objective function and one linear equality constraint. Hence, we know that every optimal solution $\bm{x}^*$ must fulfill the KKT-conditions
\begin{align*}
    &\bm{a}^\top \bm{x}^* = b, \\
    &\bm{0} \in \partial_{\bm{x}}  \mathcal L(\bm{x}^*,\lambda^*),
\end{align*}
where $\lambda^*$ is a dual optimal solution and $\mathcal L$ is the Lagrangean dual function
\[
\mathcal L(\bm{x},\lambda) = \| \bm{x}- \bm{x}_F\|_{N_1} - \lambda (\bm{a}^\top \bm{x} - b).
\]
We have $\partial_{\bm{x}}  \mathcal L(\bm{x},\lambda) = \partial f(\bm{x}-\bm{x}_F) - \lambda \bm{a}$  which proves the result. Note, that since $\bm{x}_F \neq \bm{x}_{CF}^*$ we know that $\bm{0}\notin  \partial  f(\bm{x}_{CF}^* - \bm{x}_F)$, and hence, $\lambda^*\neq 0$ must hold.
\end{proof}

\begin{proof}[\cref{thm:CF_theorem_diff}]
Without loss of generality, we assume that $\bm{x}_F$ is classified as \textit{`No'}. We apply Lemma \ref{lem:optimality_conditions_CF} to the case where the norm is differentiable. Then, there exists an $\lambda^*\neq0$ such that
\begin{align*}
    &\lambda^* \bm{a} = \nabla  f(\bm{x}_{CF}^* - \bm{x}_F),\\
    &b=\bm{a}^\top \bm{x}_{CF}^*. 
\end{align*}
By defining $\hat{\bm{a}}=\nabla  f(\bm{x}_{CF}^* - \bm{x}_F)$ and $\hat{b}=\hat{\bm{a}}^\top \bm{x}_{CF}^*$, we have $\lambda^* \bm{a}=\hat{\bm{a}}$ and $\lambda^* b=\lambda^*\bm{a}^\top \bm{x}_{CF}^*=\hat{\bm{a}}^\top \bm{x}_{CF}^*=\hat{b}$. Hence, we $\frac{1}{\lambda^*}\bm{a},b=\bm{a},b$ for a $\lambda^*\neq 0$ and the hyperplane given by $\bm{a},b$ is equivalent to the original hyperplane with parameters $\bm{a},b$ which proves the result.
\end{proof}

\begin{proof}[Lemma \ref{lem:CF_direction}]
    We note that $\bm{a}\neq\bm{0}$, hence $\|\bm{a}\|^*\neq\bm{0}$ and $d_{\bm{x}_F}$ is well defined. Besides, since we consider points $\bm{x}_F$ that do not lie on the hyperplane, we know $d_{\bm{x}_F}>0$. Moreover, we have $\|\bm{a}\|_{N_1}^*:=\sup_{\|\bm{x}\|_{N_1}\leq 1}\bm{a}^{\top}\bm{x}$. Since the unit ball defined by $\|\bm{x}\|_{N_1}\leq 1$ is a compact region and the map $\bm{x} \mapsto \bm{a}^{\top}\bm{x}$ is continuous, we know the supremum is attained at a certain point $\bm{v}$ such that $\|\bm{v}\|_{N_1}\leq 1$ and $\bm{a}^{\top}\bm{v}=\|\bm{a}\|_{N_1}^*$.

    First we consider an vector $\bm{v}$ with $\|\bm{v}\|_{N_1}\leq 1$ and $\bm{a}^{\top}\bm{v}=\|\bm{a}\|_{N_1}^*$ and show that $\bm{x}_{CF}^*= \bm{x}_{F}+d_{\bm{x}_F}\bm{v}$ is an optimal counterfactual.
    We use this $\bm{v}$ and $d_{\bm{x}_F}$ to show that the optimality conditions described in Corollary \ref{cor:optimality_conditions_CF_non_differentiable_norm} are met. For the first condition we have
    \begin{align*}
        \bm{a}^\top \bm{x}_{CF}^* &= \bm{a}^\top (\bm{x}_F+d_{\bm{x}_F}\bm{v})= \bm{a}^\top\bm{x}_F +\frac{b-\bm{a}^{\top}\bm{x}_F}{\|\bm{a}\|_{N_1}^*}\bm{a}^\top\bm{v}
        = \bm{a}^\top\bm{x}_F +\frac{b-\bm{a}^{\top}\bm{x}_F}{\|\bm{a}\|_{N_1}^*}\|\bm{a}\|_{N_1}^*\\
        &= \bm{a}^\top\bm{x}_F +b-\bm{a}^{\top}\bm{x}_F=b.
    \end{align*}
    The second and third conditions require a $\lambda\in\mathbb{R}$. Set $\lambda=\frac{\|\bm{v}\|_{N_1}}{\bm{a}^{\top}\bm{v}} \sgn(d_{\bm{x}_F})$. Then, the third condition is met:
    \[ \|\lambda\bm{a}\|_{N_1}^* = |\lambda|   \|\bm{a}\|_{N_1}^* = |\lambda|   \bm{a}^{\top}\bm{v}
    = \frac{\|\bm{v}\|_{N_1}}{|\bm{a}^{\top}\bm{v}|} \bm{a}^{\top}\bm{v}\leq \|\bm{v}\|_{N_1}\leq 1.\]
    Lastly, the second optimality condition is also satisfied:
    \begin{align*}
    \lambda \bm{a}^\top (\bm{x}_{CF}^* - \bm{x}_F) &= \lambda  d_{\bm{x}_F}  \bm{a}^\top \bm{v}
    =\frac{\|\bm{v}\|_{N_1}}{\bm{a}^{\top}\bm{v}}   \sgn(d_{\bm{x}_F})   d_{\bm{x}_F}   \bm{a}^\top \bm{v}\\
    &= |d_{\bm{x}_F}|  \|\bm{v}\|_{N_1} =\|d_{\bm{x}_F}  \bm{v}\|_{N_1} =\|\bm{x}_{CF}^* - \bm{x}_F\|_{N_1} .
    \end{align*}
    We conclude that $\bm{x}_{CF}^*= \bm{x}_{F}+d_{\bm{x}_F}\bm{v}$ is an optimal counterfactual.

    Second, we consider an optimal counterfactual $\bm{x}_{CF}^*= \bm{x}_{F}+d_{\bm{x}_F}\bm{v}$ and will show that $\bm{a}^{\top}\bm{v}=\|\bm{a}\|_{N_1}^*$ with $\|\bm{v}\|_{N_1}\leq 1$. Since $\bm{x}_{CF}^*$ is an optimal counterfactual, Corollary \ref{cor:optimality_conditions_CF_non_differentiable_norm} states that the two optimality conditions are met. The first condition tells us
    \begin{align*}
        \bm{a}^\top \bm{x}_{CF}^* &= \bm{a}^\top (\bm{x}_F+d_{\bm{x}_F}\bm{v})= \bm{a}^\top\bm{x}_F +\frac{b-\bm{a}^{\top}\bm{x}_F}{\|\bm{a}\|_{N_1}^*}\bm{a}^\top\bm{v}=b.
    \end{align*}
    Rearranging the terms yields
    \begin{align*}
        \frac{\bm{a}^\top\bm{v}}{\|\bm{a}\|_{N_1}^*}&=\frac{b-\bm{a}^{\top}\bm{x}_F}{b-\bm{a}^{\top}\bm{x}_F}=1.
    \end{align*}
    Hence, we have $\bm{a}^\top\bm{v}=\|\bm{a}\|_{N_1}^*$. The second optimality condition tells us that there exists a $\lambda\in\mathbb{R}$ such that the following relation holds:
    \begin{align*}
        \lambda \bm{a}^\top (\bm{x}_{CF}^* - \bm{x}_F) =\lambda d_{\bm{x}_F}\bm{a}^\top\bm{v}=\|\bm{x}_{CF}^* - \bm{x}_F\|_{N_1}=\|d_{\bm{x}_F}\bm{v}\|_{N_1}= |d_{\bm{x}_F}|\|\bm{v}\|_{N_1}.
    \end{align*}
    Using the fact that $\bm{a}^\top\bm{v}=\|\bm{a}\|_{N_1}^*$, we can rearrange the terms to find
    \begin{align*}
        \lambda&=\sgn(d_{\bm{x}_F})\frac{\|\bm{v}\|_{N_1}}{\|\bm{a}\|_{N_1}^*}.
    \end{align*}
    The last optimality condition implies a bound on $\|\lambda \bm a\|^*\leq 1$:
    \begin{align*}
        \|\lambda \bm a\|_{N_1}^*& = |\lambda| \|\bm{a}\|_{N_1}^* = |\lambda| \|\bm{a}\|_{N_1}^* = \frac{\|\bm{v}\|_{N_1}}{\|\bm{a}\|_{N_1}^*}\|\bm{a}\|_{N_1}^*=\|\bm{v}\|_{N_1}\leq 1.
    \end{align*}
    Therefore, it must hold that $\|\bm{v}\|_{N_1}\leq 1$ and $\bm{a}^{\top}\bm{v}=\|\bm{a}\|_{N_1}^*$
\end{proof}

\begin{proof}[Lemma \ref{lem:CF_basis}]
Note that Lemma \ref{lem:CF_direction} proves the existence of $\bm{v}$. Since the set $\{\bm{v}, \bm{v}^{2},\ldots, \bm{v}^{p}\}$ is linearly independent, we know that for $i=2,\ldots,p$ the counterfactuals are nonzero, \textit{i.e.}, $\bm{ \bm{v}^{j}}_{CF} =  \bm{v}^{j}+d_{ \bm{v}^{j}}\bm{v}\neq\bm{0}$.
    \begin{enumerate}[(i)]
        \item Assume $\bm{v}_{CF}\neq\bm{0}$. Consider $\lambda,\lambda_2,\ldots, \lambda_p$ such that
        \[\lambda\bm{v}_{CF}+\sum_{i=2}^p\lambda_i\bm{v}_{CF}^{i}=\bm{0}.\]
        By construction of the counterfactuals, we have
        \begin{align*}
        \lambda\bm{v}_{CF}+\sum_{i=2}^p\lambda_i\bm{v}_{CF}^{i} &=  \lambda\left(\bm{v}+d_{\bm{v}}\bm{v}\right)+\sum_{i=2}^p\lambda_i\left(\bm{v}^{i}+d_{\bm{v}^{i}}\bm{v}\right)\\
        &= \bm{v}\left(\lambda+\lambda d_{\bm{v}}+ \sum_{i=2}^p\lambda_i d_{\bm{v}^{i}}\right) + \sum_{i=2}^p \lambda_i\bm{v}^{i} = \bm{0}.
        \end{align*}
        This implies that $\lambda_2=\ldots=\lambda_p=0$, since otherwise $\bm{v}$ is a linear combination of vectors $\bm{v}^{2},\ldots, \bm{v}^{p}$ which contradicts with $V$ being a basis. The resulting equation is $\lambda(1+d_{\bm{v}})\bm{v}=\lambda\bm{v}_{CF}=\bm{0}$. 
        Since $\bm{v}_{CF}\neq\bm{0}$, it must hold that $\lambda=0$.
        Hence, the set $\{\bm{v}_{CF},\bm{v}_{CF}^{2},\ldots, \bm{v}_{CF}^{p}\}$ is linearly independent.
        \item Assume $\bm{v}_{CF}=\bm{0}$ and consider $\lambda_2,\ldots, \lambda_p$ such that
        \[\sum_{i=2}^p\lambda_i\bm{v}_{CF}^{i}=\bm{0}.\]
        By construction of the counterfactuals, we have
        \begin{align*}
        \sum_{i=2}^p\lambda_i\bm{v}_{CF}^{i} &= \sum_{i=2}^p\lambda_i\left(\bm{v}^{i}+d_{\bm{v}^{i}}\bm{v}\right) 
        = \sum_{i=2}^p \lambda_i\bm{v}^{i} + \bm{v} \sum_{i=2}^p\lambda_i d_{\bm{v}^{i}} = \bm{0}.
        \end{align*}
        This implies that $\lambda_2=\ldots=\lambda_p=0$, since otherwise $\bm{v}$ is a linear combination of vectors $\bm{v}^{2},\ldots, \bm{v}^{p}$ which contradicts with $V$ being a basis.
        Hence, the set $\{\bm{v}_{CF}^{2},\ldots, \bm{v}_{CF}^{p}\}$ is linearly independent.
    \end{enumerate}
\end{proof}
\begin{proof}[\cref{thm:CF_nondiff}]
It follows directly from Lemmas \ref{lem:CF_direction} and \ref{lem:CF_basis} that \cref{alg:CF_nondiff} finds $p$ linearly independent vectors on the hyperplane to retrieve the original hyperplane exactly. Since we assume $\lambda\bm{e}^{1}$ is not on the hyperplane, it takes one counterfactual query to find the direction of the counterfactuals, $\bm{v}$. Then only $p$ counterfactual queries are needed for the newly created basis $V$ to obtain a linearly independent system of equations to retrieve the original hyperplane exactly.
\end{proof}

\subsection{Robust Counterfactuals}
\begin{proof}[Lemma \ref{lem:optimality_conditions_RCF}]
Without loss of generality, we let $\bm{x}_F$ be classified as \textit{`No'} and reformulate \eqref{eq:robust_CF_problem} as
\begin{align*}
    \min_{\bm{x}_{RCF}} & \ \| \bm{x}_{RCF}-\bm{x}\|_{N_1} \\
    s.t. \quad & \bm{a}^\top (\bm{x}_{RCF} + \bm{s}) \ge b \quad \forall \bm{s}\in\mathcal S.
\end{align*}
Using the classical robust optimization reformulation, we can rewrite the infinite set of constraints as
\begin{align*}
    \bm{a}^\top (\bm{x}_{RCF} + \bm{s}) \ge b \quad \forall \bm{s}\in\mathcal S  \qquad
    &\Leftrightarrow \qquad  \bm{a}^\top \bm{x}_{RCF} + \min_{\bm{s}: \| \bm{s}\|_{N_2} \le \rho} \bm{a}^\top \bm{s} \ge b \\&
    \Leftrightarrow \qquad  \bm{a}^\top \bm{x}_{RCF} - \rho\|\bm{a}\|_{N_2}^* \ge b,
\end{align*}
where $\| \cdot \|_{N_2}^*$ is the dual norm of $\| \cdot \|_{N_2}$. The latter constraint is linear in the optimization variable $\bm{x}_{RCF}$. Since the objective function is convex, we know that every optimal solution $\bm{x}^*$ must fulfill the KKT-conditions
\begin{align*}
    & b - \bm{a}^\top \bm{x}^* + \rho\|\bm{a}\|_{N_2}^* \le 0,  \\
    &\bm{0} \in \partial_{\bm{x}}  \mathcal L(\bm{x}^*,\lambda^*), \\
    & \lambda^* (b - \bm{a}^\top \bm{x}^* + \rho\|\bm{a}\|_{N_2}^*) = 0, \\
    & \lambda^* \ge 0,
\end{align*}
where $\lambda^*$ is a dual optimal solution and $\mathcal L$ is the Lagrangian dual function
\[
\mathcal L(\bm{x},\lambda) = \| \bm{x}- \bm{x}_F\|_{N_1} + \lambda (b - \bm{a}^\top \bm{x} + \rho\|\bm{a}\|_{N_2}^*).
\]
We have $\partial_{\bm{x}}  \mathcal L(\bm{x},\lambda) = \partial f(\bm{x}-\bm{x}_F) - \lambda \bm{a}$. Note that since $\bm{x}_F \neq \bm{x}_{RCF}^*$, we have $\bm{0}\notin  \partial  f(\bm{x}_{RCF}^* - \bm{x}_F)$ and hence, $\lambda^*>0$ must hold. Applying this to the third KKT results in
\[
b - \bm{a}^\top \bm{x}^* +\rho \|\bm{a}\|_{N_2}^* = 0,
\]
which makes the first condition redundant. This proves the result.
\end{proof}

\begin{proof}[\cref{thm:RCF_diff}]
First, we query $\bm{x}_F$ such that we know the value of $q_F(\bm{x}_F)$. Then, we apply Lemma \ref{lem:optimality_conditions_RCF} to the case where the norm is differentiable. Then, there exists an $\lambda^*>0$ such that
\begin{align*}
    &\lambda^* \bm{a} = \nabla  f(\bm{x}_{RCF}^* - \bm{x}_F),\\
    &b=\bm{a}^\top \bm{x}_{RCF}^*+q_F(\bm{x}_F)\rho\|\bm{a}\|_{N_2}. 
\end{align*}
By defining $\hat{\bm{a}}=\nabla  f(\bm{x}_{RCF}^* - \bm{x}_F)$ and $\hat{b}=\hat{\bm{a}}^\top \bm{x}_{RCF}^*+q_F(\bm{x}_F)\rho\|\hat{\bm{a}}\|_{N_2}$ we have $\lambda^* \bm{a}=\hat{\bm{a}}$ and $\lambda^* b=\lambda^*\bm{a}^\top \bm{x}_{RCF}^*+\lambda^*q_F(\bm{x}_F)\rho\|\bm{a}\|_{N_2}=\hat{\bm{a}}^\top \bm{x}_{RCF}^*+q_F(\bm{x}_F)\rho\|\hat{\bm{a}}\|_{N_2}=\hat{b}$. Hence, we $\frac{1}{\lambda^*}(\hat{\bm{a}},\hat{b})=(\bm{a},b)$ for a $\lambda^*\neq 0$ and the hyperplane given by $\bm{a},b$ is equivalent to the original hyperplane with parameters $\bm{a},b$ which proves the result.
\end{proof}

\begin{proof}[Lemma \ref{lem:RCF_direction}]
    We note that $\bm{a}\neq\bm{0}$, hence $\|\bm{a}\|_{N_1}^*\neq\bm{0}$ and $d_{\bm{x}_F}$ is well defined. Besides, because $q_F(\bm{x}_F)\neq 0$ we have $d_{\bm{x}_F}>0$. Moreover, we have $\|\bm{a}\|_{N_1}^*:=\sup_{\|\bm{x}\|_{N_1}\leq 1}\bm{a}^{\top}\bm{x}$. Since the unit ball defined by $\|\bm{x}\|_{N_1}\leq 1$ is a compact space and the map $\bm{x} \mapsto \bm{a}^{\top}\bm{x}$ is continuous, we know the supremum is attained at a certain point $\bm{v}$ such that $\|\bm{v}\|_{N_1}\leq 1$ and $\bm{a}^{\top}\bm{v}=\|\bm{a}\|_{N_1}^*$.

    First, we consider a vector $\bm{v}$ with $\|\bm{v}\|_{N_1}\leq 1$ and $\bm{a}^{\top}\bm{v} = \|\bm{a}\|_{N_1}^*$ and show that $\bm{x}_{RCF}^*=\bm{x}_F+d_{\bm{x}_F}\bm{v}$ is an optimal robust counterfactual. We show the optimality conditions described in Corollary \ref{cor:optimality_conditions_CF_non_differentiable_norm} are met. For the first condition we have
    \begin{align*}
        \bm{a}^\top \bm{x}_{RCF}^* &= \bm{a}^\top (\bm{x}_F+d_{\bm{x}_F}\bm{v})= \bm{a}^\top\bm{x}_F +\left(\frac{b-\bm{a}^{\top}\bm{x}_F- q_F(\bm{x}_F)\rho\|\bm{a}\|_{N_2}^*}{\|\bm{a}\|_{N_1}^*} \right)\bm{a}^\top\bm{v}\\
        &= \bm{a}^\top\bm{x}_F +b-\bm{a}^{\top}\bm{x}_F - q_F(\bm{x}_F)\rho\|\bm{a}\|_{N_2}^*\\
        &=b - q_F(\bm{x}_F)\rho\|\bm{a}\|_{N_2}^*.
    \end{align*}
    The second and third conditions require a $\lambda\in\mathbb{R}$. Set $\lambda=\frac{\|\bm{v}\|_{N_1}}{\bm{a}^{\top}\bm{v}} \sgn(d_{\bm{x}_F})$. Then, the third condition is met:
    $$\|\lambda\bm{a}\|_{N_1}^*=|\lambda| \|\bm{a}\|_{N_1}^*=|\lambda| \bm{a}^{\top}\bm{v}
    = \frac{\|\bm{v}\|_{N_1}}{|\bm{a}^{\top}\bm{v}|} \bm{a}^{\top}\bm{v}\leq \|\bm{v}\|_{N_1}\leq 1.$$
    Lastly, the second optimality condition is also met:
    \begin{align*}
    \lambda \bm{a}^\top (\bm{x}_{RCF}^* - \bm{x}_F) &= \lambda  d_{\bm{x}_F}  \bm{a}^\top \bm{v}
    =\frac{\|\bm{v}\|_{N_1}}{\bm{a}^{\top}\bm{v}} \sgn(d_{\bm{x}_F})   d_{\bm{x}_F}   \bm{a}^\top \bm{v}\\
    &= |d_{\bm{x}_F}|  \|\bm{v}\|_{N_1} =\|d_{\bm{x}_F}  \bm{v}\|_{N_1}=\|\bm{x}_{RCF}^* - \bm{x}_F\|_{N_1} .
    \end{align*}
    We conclude that $\bm{x}_{RCF}^*=\bm{x}_F+d_{\bm{x}_F}\bm{v}$ is an optimal robust counterfactual.

    Second, we consider an optimal counterfactual $\bm{x}_{RCF}^*= \bm{x}_{F}+d_{\bm{x}_F}\bm{v}$ and show that $\bm{a}^{\top}\bm{v}=\|\bm{a}\|_{N_1}^*$ with $\|\bm{v}\|_{N_1}\leq 1$. Since $\bm{x}_{RCF}^*$ is an optimal counterfactual, Corollary \ref{cor:optimality_conditions_RCF_non_differentiable_norm} implies the corresponding optimality conditions are met. The first condition tells us
    \begin{align*}
        \bm{a}^\top \bm{x}_{RCF}^* &= \bm{a}^\top (\bm{x}_F+d_{\bm{x}_F}\bm{v})= \bm{a}^\top\bm{x}_F + \left(\frac{b-\bm{a}^{\top}\bm{x}_F - q_F(\bm{x}_F) \rho \|\bm{a}\|_{N_2}^*}{\|\bm{a}\|_{N_1}^*} \right) \bm{a}^\top \bm{v}\\
        &= b - q_F(\bm{x}_F) \rho \|\bm{a}\|_{N_2}^*.
    \end{align*}
    Rearranging the terms yields
    \begin{align*}
        \frac{\bm{a}^\top\bm{v}}{\|\bm{a}\|_{N_1}^*}&=\frac{b-\bm{a}^{\top}\bm{x}_F - q_F(\bm{x}_F) \rho \|\bm{a}\|_{N_2}^*}{b-\bm{a}^{\top}\bm{x}_F - q_F(\bm{x}_F) \rho \|\bm{a}\|_{N_2}^*}=1.
    \end{align*}
    Hence, we have $\bm{a}^\top\bm{v}=\|\bm{a}\|_{N_1}^*$. The second optimality condition tells us there exists a $\lambda\in\mathbb{R}$ such that the following relation holds.
    \begin{align*}
        \lambda \bm{a}^\top (\bm{x}_{RCF}^* - \bm{x}_F) =\lambda d_{\bm{x}_F}\bm{a}^\top\bm{v}=\|\bm{x}_{RCF}^* - \bm{x}_F\|_{N_1}=\|d_{\bm{x}_F}\bm{v}\|_{N_1}= |d_{\bm{x}_F}| \|\bm{v}\|_{N_1}\\
    \end{align*}
    Using the fact that $\bm{a}^\top\bm{v}=\|\bm{a}\|_{N_1}^*$, we can rearrange the terms to find
    \begin{align*}
        \lambda&=\sgn(d_{\bm{x}_F})\frac{\|\bm{v}\|_{N_1}}{\|\bm{a}\|_{N_1}^*}.
    \end{align*}
    The last optimality condition implies a bound on $\|\lambda \bm a\|^*\leq 1$.
    \begin{align*}
        \|\lambda \bm a\|_{N_1}^*& = |\lambda|   \|\bm{a}\|_{N_1}^* = |\lambda| \|\bm{a}\|_{N_1}^* = \frac{\|\bm{v}\|_{N_1}}{\|\bm{a}\|_{N_1}^*}\|\bm{a}\|_{N_1}^*=\|\bm{v}\|_{N_1}\leq 1.
    \end{align*}
    Therefore, it must hold that $\|\bm{v}\|_{N_1}\leq 1$ and $\bm{a}^{\top}\bm{v}=\|\bm{a}\|_{N_1}^*$.
\end{proof}

\begin{proof}[Lemma \ref{lem:perspective_point}]
Without loss of generality, assume that $\bm{x}_F$ is classified as \textit{`No'}. From Corollary \ref{cor:optimality_conditions_RCF_non_differentiable_norm} it follows that there exists a $\lambda\in\mathbb R\setminus \{ 0\}$ such that
\begin{align}
    &\bm{a}^\top \bm{x}_{RCF}^*  = b + \rho \|\bm{a}\|_{N_2}^*, \label{eq:proof_perspective_1}\\
    &  \lambda \bm{a}^\top \bm{v} = 1, \label{eq:proof_perspective_2}\\
    & \|\lambda \bm{a} \|_{N_1}^* \le 1 .\label{eq:proof_perspective_3}
\end{align}
Furthermore, from the second condition the definition of $\bm{v}$ and since $\bm{a}^\top \bm{x}_{RCF}^*\ge b \ge \bm{a}^\top \bm{x}_{F}$, it follows that $\lambda>0$ must hold. 
From the latter conditions and the equivalence of the norms, we obtain for the inner product of $\bm{a}$ and $\bar{\bm{x}}$
\begin{equation}\label{eq:proof_lemma_perspective_point}
\begin{aligned}
    \bm{a}^{\top}\bar{\bm{x}} & = \bm{a}^\top\left( \bm{x}_{RCF}^* - d\bm{v} \right)\\
    & = b+\rho\|\bm{a}\|_{N_2}^*- d\lambda^{-1} \\
    & \le  b+\rho C\|\bm{a}\|_{N_1}^*- d\lambda^{-1} \\
    & = b + \lambda^{-1}\left( \rho C -d\right),
    \end{aligned}
\end{equation}
where the second equality follows from \eqref{eq:proof_perspective_1} and \eqref{eq:proof_perspective_2}, the first inequality follows from the equivalence of the dual norms, and the last equality follows from \eqref{eq:proof_perspective_2} and \eqref{eq:proof_perspective_3} since $\lambda^{-1} = \bm{a}^\top \bm{v} \le \|\bm{a}\|_{N_1}^* \le \lambda^{-1}$. Hence for $d\ge \rho C$ we have $\bm{a}^{\top}\bar{\bm{x}}\le b$.
\end{proof}

\begin{proof}[Corollary \ref{cor:N1N2}]
The statements (i) and (ii) follow from applying Lemma \ref{lem:perspective_point} together with classical results for dual $\ell_p$-norms and equivalence between $\ell_p$-norms. Statement (iii) follows since the inequality in \eqref{eq:proof_lemma_perspective_point} is an equality with $C=1$, if $N_1=N_2$.
\end{proof}

\end{document}